\documentclass[reqno,11pt]{amsart}
\usepackage{tikz}
\usetikzlibrary{arrows.meta}
\usetikzlibrary{decorations.pathmorphing}
\usepackage{esint}
\usepackage{verbatim}
\usepackage{graphicx}
\usepackage{subfigure}
\usepackage{amsfonts}
\usepackage{amssymb}
\usepackage{amsmath}
\usepackage{amsthm}
\usepackage{latexsym}
\usepackage{amscd}
\usepackage{mathrsfs}
\usepackage{verbatim}
\usepackage{mathrsfs}
\usepackage{enumitem}
\usepackage{braket}
\usepackage{float}
\usepackage[margin=1.3in]{geometry}
\usepackage{comment}
\usepackage{bm}
\usepackage{xcolor}
\usepackage{imakeidx}
\makeindex[title=Index of Notation]

\usepackage{hyperref}
\hypersetup{
	colorlinks=true,
	linkcolor=black,
	 urlcolor=black,
	citecolor=black
}
\usepackage{caption}
\usepackage{array}

\usepackage{amssymb}
\usepackage{pdfpages}
\usepackage{verbatim}

\usepackage{appendix}
\usepackage{xspace}

\newtheorem{theorem}{Theorem}[section]
\newtheorem{conjecture}[theorem]{Conjecture}
\newtheorem{lemma}[theorem]{Lemma}
\newtheorem{proposition}[theorem]{Proposition}
\newtheorem{corollary}[theorem]{Corollary}
\theoremstyle{definition}
\newtheorem{definition}[theorem]{Definition}

\theoremstyle{remark}
\newtheorem{remark}[theorem]{Remark}

\newcommand{\IN}{\mathbb{N}}
\newcommand{\IR}{\mathbb{R}}
\newcommand{\IS}{\mathbb{S}}

\newcommand{\CB}{\mathcal{B}}
\newcommand{\CC}{\mathcal{C}}

\newcommand{\CG}{\mathcal{G}}
\newcommand{\CH}{\mathcal{H}}

\newcommand{\CM}{\mathcal{M}}
\newcommand{\CN}{\mathcal{N}}
\newcommand{\CS}{\mathcal{S}}

\newcommand{\CY}{\mathcal{Y}}

\numberwithin{equation}{section}

\begin{document}
	\title[Bounded mean curvature and Morse index]{ Singularities of mean curvature flows with bounded mean curvature and Morse index}
	
	\author{Yongheng Han}
	\address{Institute of Mathematics, AMSS, CAS, Beijing 100190, China}
	
	\email{hyh98@amss.ac.cn}

	\date{\today}

	
	\begin{abstract}
	
  We study the multiplicity of the singularities of mean curvature flows with bounded mean curvature and Morse index. For $3\leq n\leq 6$, we show that  either the mean curvature or the Morse index blows up at the first singular
  time for a closed smooth embedded mean curvature flow in $\mathbb{R}^{n+1}$.
	\end{abstract}
	
\maketitle
%\tableofcontents
\section{Introduction}
Let $\mathbf{x}_0:M^n\to \mathbb{R}^{n+1}$ be a closed hypersurface in $\mathbb{R}^{n+1}$. A one-parameter family of immersions $\mathbf{x}(p,t):M\to  \mathbb{R}^{n+1}$ is called a mean
curvature flow if $\mathbf{x}$ satisfies the equation
\begin{equation}\label{eq0.1}
	\frac{\partial \mathbf{x}}{\partial t}=-H\mathbf{n},\quad \mathbf{x}(0)=\mathbf{x}_0,
\end{equation}
where $H$ denotes the mean curvature of the hypersurface $M_t:=\mathbf{x}(t)(M)$ and $\mathbf{n}$ denotes  the unit normal vector field of $M_t$. Huisken \cite{Hui90} proved that if the flow \eqref{eq0.1} develops a singularity at time $T<\infty$, then the second fundamental form will blow up at time $T$. In his book, Mantegazza \cite{Man11} asked, “Is a uniform bound on the mean curvature $H$ alone sufficient to prevent singularities during the flow?” Recently, Vishwamitra \cite{Vis26} showed that if both 
$H$ and $\nabla H$ are uniformly bounded, the flow can be extended.

In dimension two, Ilmanen \cite{Ilm26} showed that for a smooth embedded mean curvature flow in 
$\mathbb{R}^3$, the tangent flow at the first singular time must be smooth, and he conjectured that the convergence is of multiplicity one. In \cite{LW19,LW22}, Li and Wang showed that if the mean curvature is of type-I, then the convergence is of multiplicity one. As a corollary, they showed that for a closed smooth embedded mean curvature flow in 
$\mathbb{R}^3$, the mean curvature blows up at the first finite singular time, thereby giving an affirmative answer to Mantegazza’s question. Recently, Bamler and Kleiner \cite{BK23} proved the Multiplicity One Conjecture in full generality.

For $n\geq 7$, Vel\'azquez  \cite{Vel94} constructed a family of smooth embedded hypersurfaces evolving in $\IR^{n+1}$ whose blowup is a minimizing cone. Moreover, Stolarski \cite{Sto23} established the existence of mean curvature flows with uniformly bounded mean curvature, thus answering Mantegazza's question in the negative.  For further studies of mean curvature flows modeled by minimizing cone, we refer the reader to  \cite{ADS26,GS18,Liu24,Hua25}.

For $3\leq n\leq 6$, no progress has been made on this problem. Motivated by Vel\'azquez's examples, Ilmanen \cite{Ilm26} conjectured that, for $n \le 6$, all tangent flows of embedded mean curvature flows in $\IR^{n+1}$ are smooth. If both this conjecture and the multiplicity one conjecture (which asserts that if the entropy of the initial hypersurface is less than two, then the tangent flow is of multiplicity one) hold, then for any closed smooth embedded mean curvature flow in $\IR^{n+1}$ with $n\leq 6$, the mean curvature would blow up at the first finite singular time. This motivates the following conjecture.
\begin{conjecture}\label{conj0.3}
	The mean curvature blows up at the first finite singular time for a closed smooth embedded mean curvature flow in $\IR^{n+1}$ with $3\leq n\leq 6$.
\end{conjecture}
 In this paper, we study the multiplicity of the singularities of mean curvature flows with bounded mean curvature and bounded index. We show that 
\begin{theorem}[Main theorem]\label{thm0.1}
For $n\ge 3$, let $\mathbf{x}:M^n\times [0,T)\to \mathbb{R}^{n+1}$ be a closed smooth embedded
	mean curvature flow in $\mathbb{R}^{n+1}$. Suppose that
	\begin{equation}\label{eq0.2}
			\sup\limits_{M\times [0,T)}|H|(x,t)=\Lambda<\infty
	\end{equation}
	and 
	\begin{equation}\label{eq0.3}
			\sup\limits_{t\in [0,T)} \mathrm{index}(M_t)=I<\infty ,
	\end{equation}
		where $\mathrm{index}(M_t)$ is the number of negative eigenvalues of the operator $\Delta+|A|^2$ on $M_t$.
	Then there exist a smooth, possibly incomplete, hypersurface $M_T$ and a set $\mathcal{S}:=\overline{M_T}\setminus M_T$ satisfying the following properties.
    \begin{enumerate}
       \item [(1)]$M_t$ converges smoothly to $M_T$ with multiplicity one. 
        \item [(2)] $\mathcal{S}$ has Minkowski dimension $\leq n-7$.
    In particular, for $3\leq n\leq 6$, $M_t$ does not blow up at time $T$. 

    \item [(3)]For any $p\in \mathcal{S}$ and any sequence of time $t_i\to T$, $\frac{1}{\sqrt{T-t_i}}(M_{t_i}-p)$ (after taking a subsequence) converges to a stable  minimal cone with multiplicity one.
    \end{enumerate}
\end{theorem}
\begin{remark}
    Though many examples of mean curvature flows with bounded mean curvature are known, they all appear to have uniformly bounded Morse index. This leads us to conjecture that bounded mean curvature actually implies bounded Morse index.
\end{remark}

	 In \cite{Sto25}, Stolarski proved that every tangent flow 
is a static stationary cone. Moreover, tangent flows at a space-time point $(x_0,t_0)$ are in one-to-one correspondence with the tangent cones of the 
$t_0$-time slice at $x_0$. Thus, the space-time singularity structure of the flow is completely determined by the tangent cones of its time slices. Using his result, we derive the following corollary.
\begin{corollary}\label{coro0.2}
	Let $\mathbf{x}:M^7\times [0,T)\to \mathbb{R}^{8}$ be a closed smooth embedded
	mean curvature flow. Suppose that
	\begin{equation*}
		\sup\limits_{M\times [0,T)}|H|(x,t)<\infty\quad \text{and}\quad\sup\limits_{t\in [0,T)} \mathrm{index}(M_t)<\infty.
	\end{equation*}
		Then,  there exist a smooth (possibly non-complete) hypersurface $M_T$ and a discrete set $\mathcal{S} := \overline{M_T}\setminus M_T$ such that $M_t$ converges smoothly to $M_T$ away from   $\CS$ with multiplicity one. Moreover, for any $p\in\mathcal{S}$ there is a stable regular cone $\Sigma$ such that  $\frac{1}{\sqrt{T-t}}(M_{t}-p)$ converges smoothly to  $\Sigma$ away from the origin (see Figure \ref{fig:stable-cone}).% and every limit flow is  a static regular cone.
\end{corollary}

\begin{figure}[htbp]
    \centering
   \resizebox{0.4\textwidth}{!}{

\tikzset{every picture/.style={line width=0.75pt}} %set default line width to 0.75pt        

\begin{tikzpicture}[x=0.75pt,y=0.75pt,yscale=-1,xscale=1]
%uncomment if require: \path (0,428); %set diagram left start at 0, and has height of 428

%Curve Lines [id:da1655011221236834] 
\draw    (104.5,36) .. controls (144.5,6) and (217.5,26) .. (234.5,67) ;
%Curve Lines [id:da0949263068380638] 
\draw    (87.5,188) .. controls (71.5,161) and (31.5,137) .. (56.9,90.53) .. controls (82.31,44.06) and (77.22,56.46) .. (104.5,36) ;
%Curve Lines [id:da519291729093088] 
\draw    (87.5,188) .. controls (141.1,251.59) and (174.5,245) .. (172.5,321) .. controls (170.5,397) and (188.19,297.13) .. (203.88,244.11) .. controls (219.57,191.09) and (278.5,182) .. (280.5,146) ;
%Curve Lines [id:da15009168934091754] 
\draw    (234.5,67) .. controls (256,124) and (293.5,81) .. (280.5,146) ;
%Shape: Arc [id:dp30823639564456007] 
\draw  [draw opacity=0][dash pattern={on 4.5pt off 4.5pt}] (112.45,100.02) .. controls (117.47,90.5) and (127.48,84) .. (139,84) .. controls (142.6,84) and (146.06,84.64) .. (149.26,85.8) -- (139,114) -- cycle ; \draw  [dash pattern={on 4.5pt off 4.5pt}] (112.45,100.02) .. controls (117.47,90.5) and (127.48,84) .. (139,84) .. controls (142.6,84) and (146.06,84.64) .. (149.26,85.8) ;  
%Shape: Arc [id:dp04519300824642536] 
\draw  [draw opacity=0] (150.46,70.27) .. controls (151.11,73.62) and (151.19,77.13) .. (150.63,80.69) .. controls (148.04,97.06) and (132.67,108.22) .. (116.31,105.63) .. controls (110.47,104.71) and (105.29,102.16) .. (101.17,98.51) -- (121,76) -- cycle ; \draw   (150.46,70.27) .. controls (151.11,73.62) and (151.19,77.13) .. (150.63,80.69) .. controls (148.04,97.06) and (132.67,108.22) .. (116.31,105.63) .. controls (110.47,104.71) and (105.29,102.16) .. (101.17,98.51) ;  
%Shape: Arc [id:dp8410412173775886] 
\draw  [draw opacity=0] (239.69,117.06) .. controls (233.75,125.63) and (223.51,130.86) .. (212.39,129.89) .. controls (201.36,128.92) and (192.25,122.11) .. (187.85,112.78) -- (215,100) -- cycle ; \draw   (239.69,117.06) .. controls (233.75,125.63) and (223.51,130.86) .. (212.39,129.89) .. controls (201.36,128.92) and (192.25,122.11) .. (187.85,112.78) ;  
%Shape: Arc [id:dp33584152128319467] 
\draw  [draw opacity=0][dash pattern={on 4.5pt off 4.5pt}] (194.94,119.05) .. controls (199.37,116.47) and (204.51,115) .. (210,115) .. controls (218.28,115) and (225.78,118.36) .. (231.21,123.79) -- (210,145) -- cycle ; \draw  [dash pattern={on 4.5pt off 4.5pt}] (194.94,119.05) .. controls (199.37,116.47) and (204.51,115) .. (210,115) .. controls (218.28,115) and (225.78,118.36) .. (231.21,123.79) ;  
%Shape: Arc [id:dp35360731947550783] 
\draw  [draw opacity=0][dash pattern={on 4.5pt off 4.5pt}] (132.45,264.4) .. controls (144.97,261.04) and (161.45,259) .. (179.5,259) .. controls (200.27,259) and (218.96,261.7) .. (231.94,266.01) -- (179.5,280.5) -- cycle ; \draw  [dash pattern={on 4.5pt off 4.5pt}] (132.45,264.4) .. controls (144.97,261.04) and (161.45,259) .. (179.5,259) .. controls (200.27,259) and (218.96,261.7) .. (231.94,266.01) ;  
%Shape: Arc [id:dp4613415028053858] 
\draw  [draw opacity=0][dash pattern={on 4.5pt off 4.5pt}] (149.57,323.98) .. controls (157.03,322.07) and (165.4,321) .. (174.25,321) .. controls (185.74,321) and (196.43,322.81) .. (205.37,325.91) -- (174.25,351) -- cycle ; \draw  [dash pattern={on 4.5pt off 4.5pt}] (149.57,323.98) .. controls (157.03,322.07) and (165.4,321) .. (174.25,321) .. controls (185.74,321) and (196.43,322.81) .. (205.37,325.91) ;  
%Straight Lines [id:da5327316389181005] 
\draw    (199,288) -- (414.57,231.51) ;
\draw [shift={(416.5,231)}, rotate = 165.31] [color={rgb, 255:red, 0; green, 0; blue, 0 }  ][line width=0.75]    (10.93,-3.29) .. controls (6.95,-1.4) and (3.31,-0.3) .. (0,0) .. controls (3.31,0.3) and (6.95,1.4) .. (10.93,3.29)   ;
%Straight Lines [id:da03350888704239474] 
\draw    (389.5,99) -- (501.5,346) ;
%Straight Lines [id:da5989994371020163] 
\draw    (501.5,346) -- (609.45,98.98) ;
%Shape: Arc [id:dp09204820318023288] 
\draw  [draw opacity=0][dash pattern={on 4.5pt off 4.5pt}] (389.5,99) .. controls (389.48,98.83) and (389.48,98.65) .. (389.48,98.48) .. controls (389.48,84.95) and (438.72,73.98) .. (499.48,73.98) .. controls (560.23,73.98) and (609.48,84.95) .. (609.48,98.48) .. controls (609.48,98.65) and (609.47,98.81) .. (609.45,98.98) -- (499.48,98.48) -- cycle ; \draw  [dash pattern={on 4.5pt off 4.5pt}] (389.5,99) .. controls (389.48,98.83) and (389.48,98.65) .. (389.48,98.48) .. controls (389.48,84.95) and (438.72,73.98) .. (499.48,73.98) .. controls (560.23,73.98) and (609.48,84.95) .. (609.48,98.48) .. controls (609.48,98.65) and (609.47,98.81) .. (609.45,98.98) ;  
%Shape: Arc [id:dp6238916588388692] 
\draw  [draw opacity=0] (605.74,100.49) .. controls (602.08,116.12) and (555.9,128.48) .. (499.48,128.48) .. controls (440.66,128.48) and (392.98,115.05) .. (392.98,98.48) -- (499.48,98.48) -- cycle ; \draw   (605.74,100.49) .. controls (602.08,116.12) and (555.9,128.48) .. (499.48,128.48) .. controls (440.66,128.48) and (392.98,115.05) .. (392.98,98.48) ;  
%Straight Lines [id:da07952420319260112] 
\draw    (495.5,79) -- (495.5,124) ;
\draw [shift={(495.5,126)}, rotate = 270] [color={rgb, 255:red, 0; green, 0; blue, 0 }  ][line width=0.75]    (10.93,-3.29) .. controls (6.95,-1.4) and (3.31,-0.3) .. (0,0) .. controls (3.31,0.3) and (6.95,1.4) .. (10.93,3.29)   ;
%Curve Lines [id:da8433860794932448] 
\draw    (388.5,124) .. controls (428.5,94) and (464.5,359) .. (504.5,329) ;
%Curve Lines [id:da8900101814482704] 
\draw    (504.5,329) .. controls (544.5,299) and (568.5,142) .. (608.5,112) ;
%Shape: Arc [id:dp17485535576267142] 
\draw  [draw opacity=0][dash pattern={on 4.5pt off 4.5pt}] (444.14,302.93) .. controls (446.15,295.15) and (470.06,289) .. (499.25,289) .. controls (528.6,289) and (552.6,295.21) .. (554.39,303.06) -- (499.25,304) -- cycle ; \draw  [dash pattern={on 4.5pt off 4.5pt}] (444.14,302.93) .. controls (446.15,295.15) and (470.06,289) .. (499.25,289) .. controls (528.6,289) and (552.6,295.21) .. (554.39,303.06) ;  
%Shape: Arc [id:dp8162295980397541] 
\draw  [draw opacity=0][dash pattern={on 4.5pt off 4.5pt}] (389.14,161.77) .. controls (400.81,148.19) and (449.81,138) .. (508.5,138) .. controls (562,138) and (607.45,146.47) .. (623.91,158.25) -- (508.5,168) -- cycle ; \draw  [dash pattern={on 4.5pt off 4.5pt}] (389.14,161.77) .. controls (400.81,148.19) and (449.81,138) .. (508.5,138) .. controls (562,138) and (607.45,146.47) .. (623.91,158.25) ;  
%Straight Lines [id:da6719172141637931] 
\draw    (56.5,268) -- (89.48,212.72) ;
\draw [shift={(90.5,211)}, rotate = 120.82] [color={rgb, 255:red, 0; green, 0; blue, 0 }  ][line width=0.75]    (10.93,-3.29) .. controls (6.95,-1.4) and (3.31,-0.3) .. (0,0) .. controls (3.31,0.3) and (6.95,1.4) .. (10.93,3.29)   ;

% Text Node
\draw (276,253.4) node [anchor=north west][inner sep=0.75pt]  [font=\Huge]  {$\frac{1}{\sqrt{T-t}}$};
% Text Node
\draw (419,20.4) node [anchor=north west][inner sep=0.75pt]  [font=\Huge]  {$stable\ cone$};
% Text Node
\draw (32,270.4) node [anchor=north west][inner sep=0.75pt]  [font=\Huge]  {$M_{t}$};

\end{tikzpicture}

}
     \caption{Blow-up near a singularity.}
    \label{fig:stable-cone}
\end{figure}

Our proof of the multiplicity one theorem is motivated by the compactness theory for self-shrinkers.
Consider a sequence of shrinking surfaces with uniformly bounded area and genus. If a subsequence converges to a smooth shrinking surface with multiplicity greater than one, then Colding and Minicozzi~\cite{CM13} constructed a positive Jacobi function for the $L$-operator(see \eqref{eq5.12} for the definition) by renormalizing the height function. 
Using this Jacobi function, they showed that the self-shrinker must be stable, which contradicts Lemma~\ref{lm-b}. 

In the flow setting, we similarly construct a positive immortal solution to $\partial_t u - Lu = 0$ via the height function, relying crucially on the long-time pseudolocality theorem. In dimension two, Li and Wang~\cite{LW19} relied on the Gauss--Bonnet formula. Since no such formula exists in higher dimensions, we circumvent this issue by assuming that the flow has uniformly bounded Morse index.
Recently, Brendle and Tsiamis~\cite{B25} established a compactness result by obtaining a positive lower bound for the first eigenvalue; this raises the natural question of whether their method can be generalized to the mean curvature flow.

To that end, we need suitable estimates for the renormalized function.
When the multiplicity of convergence exceeds one, we show that the volume between the top and bottom sheets tends to zero. 
Unlike in the work of Li and Wang, where subsequences always converge to smooth shrinkers, we need to handle singularities.
We then show that the area is comparable to the $L^1$-norm of the height function; this step is directly from Li and Wang's argument. Next, combining the almost monotonicity of the $L^1$-norm with Harnack's inequality, we obtain a pointwise upper bound for the function at any time. 

$\\$
\textbf{Structure of this paper:} Section~\ref{sec2} recalls some basic facts about integral varifolds and mean curvature flow. In Section~\ref{sec3}, we establish a new two-sided pseudolocality theorem for mean curvature flow and subsequently develop the corresponding weak compactness theory. Section~\ref{sec4} demonstrates that the rescaled mean curvature flow converges to a stationary cone with multiplicity one, which allows us to complete the proof of Theorem~\ref{thm0.1}.
$\\$

 \noindent
\textbf{Acknowledgment:}
The author is grateful to his advisor, Bing Wang, for his guidance and support.
 The author is supported by  the Project of Stable Support for Youth Team in
 Basic Research Field,  Chinese Academy of Sciences,  (YSBR-001).  The authors would
like to thank the anonymous reviewers for their insightful comments and constructive
suggestions.

\section{Preliminaries}\label{sec2}
\subsection{Integral Varifolds and Mean Curvature}$\\$

We begin by recalling the notion of integral varifolds.

\begin{definition}
	Let $(N^n,g)$ be a smooth Riemannian manifold. An integral varifold $V$ of dimension $m$ in $N$ is a pair $(M,\theta)$, where $M\subset N$ is a countably $m$-rectifiable set and $\theta:M\to \mathbb Z_{>0}$ is a locally integrable function. We associate to $V$ the measure
	\begin{equation*}
		\begin{split}
			\mu_V(A)=\int_{M\cap A}\theta\, d\CH^m \qquad \text{for every Borel set } A\subset N,
		\end{split}
	\end{equation*}
	where $\CH^m$ denotes the $m$-dimensional Hausdorff measure. The mass of $V$ is defined as $\mathbf M(V):=\mu_V(N)$. We denote by $\mathbf{IV}_m(N)$ the space of all $m$-dimensional integral varifolds in $N$.
\end{definition}

Next, we recall the definitions of stationarity and generalized mean curvature.

\begin{definition}
	Let $V$ be an $m$-varifold in $N$ and let $X\in C_c^1(N,TN)$. The first variation of $V$ along $X$ is defined by
	\begin{equation*}
		\begin{split}
			\delta V(X)=\frac{d}{dt}\bigg|_{t=0}\mathbf M((\Phi_t)_\sharp V),
		\end{split}
	\end{equation*}
	where $\Phi_t$ is the family of diffeomorphisms generated by $X$. We say that $V$ has bounded generalized mean curvature if there exists a constant $C\ge 0$ such that
	\begin{equation*}
		\begin{split}
			|\delta V(X)|\leq C\int |X|\,d\mu_V \qquad \text{for all } X\in C_c^1(N,TN). 
		\end{split}
	\end{equation*}
	If $\delta V(X)=0$ for every $X\in C_c^1(N,TN)$, then $V$ is called stationary in $N$.
\end{definition}

Whenever $V$ has bounded generalized mean curvature, the Riesz representation theorem together with the Radon–Nikodym theorem yields a vector field $H\in L^1_{\mathrm{loc}}(\mu_V;TN)$, defined $\mu_V$-almost everywhere on $M$, such that
\begin{equation*}
	\begin{split}
		\delta V(X)=-\int X\cdot H\,d\mu_V \qquad \text{for all } X\in C_c^1(N,TN). 
	\end{split}
\end{equation*}

We will make use of the following fundamental compactness result due to Allard \cite{All72}, which states that any sequence of integral varifolds with uniformly bounded mass and mean curvature admits a convergent subsequence.

\begin{lemma}[Allard compactness theorem]
	Let $\{V_i\}$ be a sequence of integral $m$-varifolds in an open set $\mathcal{U}\subset(N,g)$ such that $\mathbf M(V_i\cap \mathcal{U})\le \mu_0$ and their generalized mean curvatures are uniformly bounded by a constant $\Lambda$. Then after passing to a subsequence, $V_i$ converges in the varifold sense to an integral varifold $V$ satisfying $\mathbf M(V\cap \mathcal{U})\le \mu_0$, whose generalized mean curvature is also bounded by $\Lambda$.
\end{lemma}

Allard also proved a regularity theorem that provides a local graphical description whenever the mean curvature is small in a suitable $L^p$ sense and the density is close to that of a plane.

\begin{lemma}[Allard regularity theorem]\label{allard}
Let $V$ be a sequence of integral $n$-varifolds in $\IR^{n+1}$.	If $p>n$ is arbitrary, then there are $\gamma=\gamma(n,k,p)$,$\lambda_0=\lambda_0(n,k,p)\in (0,\frac{1}{16}]$ such that if $\lambda\in (0,\lambda_0]$ and if the following hypotheses hold
	\begin{equation*}
		\omega_n^{-1}\rho^{-n}\mu_V(B_\rho(0))\leq 1+\lambda,\quad \bigg(\rho^{p-n}\int_{B_{\rho}(0)}|H|^p\, d\mu_V\bigg)^{1/p}\leq \lambda.
	\end{equation*}
	then there is an orthogonal transformation $Q$ of $\mathbb{R}^{n+k}$ and a $u=(u^1,\cdots, u^k)\in C^{1,1-n/p}$ such that  $Du(0)=0$,  $\mathrm{supp}V\cap B_{\gamma\rho}(0)=Q(\mathrm{graph}u)\cap B_{\gamma\rho}(0)$, and 
	\begin{equation*}
		\begin{split}
			\rho^{-1}\sup |u|+\sup |Du|+\rho^{1-\frac{n}{p}}\sup\limits_{x,y\in B_{\gamma\rho}(0),x\neq y}\frac{|Du(x)-Du(y)|}{|x-y|^{1-n/p}}\leq C\lambda^{\frac{1}{2n+2}},
		\end{split}
	\end{equation*}
	where $C=C(n,k,p)>0$ and $\gamma=\gamma(n,k,p)\in (0,1)$.
\end{lemma}

\subsection{The second variation of stationary varifold}$\\$

If $M$ is a $C^2$ embedded hypersurface (possibly noncomplete) in $N$, we define its regular and singular sets as follows. Let $\operatorname{reg} M$ denote the set of points $p\in \overline M$ for which there exists $\rho>0$ such that $B_\rho(p)\cap \overline M$ is a connected embedded $C^2$ hypersurface. The singular set is then defined by
\begin{equation*}
    \operatorname{sing} M :=\overline M\setminus \operatorname{reg} M.
\end{equation*}
We shall always identify $M$ with its regular part, i.e., $M=\operatorname{reg} M$.

Now assume that $M\in  \mathbf{IV}_m(N)$ (i.e., $M$ is an $m$-dimensional integral varifold in $N$). In this setting we define
\begin{align*}
    \operatorname{reg} M := \Bigl\{
        x\in M :{}& \text{\(B_\varepsilon(x)\cap \operatorname{supp}M\) is an embedded, connected} \\
        & \text{\(C^{1,\alpha}\) manifold for some \(\varepsilon>0\)}
    \Bigr\},
\end{align*}
and set $\operatorname{sing}M:=\operatorname{spt}M\setminus \operatorname{reg}M$. Throughout this paper we restrict our attention to those $M$ satisfying $\mathcal H^{n-2}(\operatorname{sing}M)=0$ and $\theta(x)=1$ for $\mathcal H^n$-a.e. $x\in \operatorname{spt}M$.

Suppose now that $M^n$ is a smooth minimal surface in $N^{n+1}$. Then the second variation of mass is given by
\begin{equation*}
\begin{split}
\delta^2(X)&:=\frac{d^2}{dt^2}\bigg|_{t=0}\mathbf M((\Phi_t)_\sharp M)\\
&=\int_M\left(|\nabla^\perp X|^2 - A(X,X) - \operatorname{Ric}(X,X)\right)\,d\mathcal H^n
\end{split}
\end{equation*}
for all $X\in C_c^1(N,TN)$, where $A$ denotes the shape operator. For a hypersurface with trivial normal bundle, we may identify a normal vector field $X=f\nu$, and the second variation becomes
\[
\delta^2(X)=\int_M\left(|\nabla f|^2 - |A|^2 f^2 - \operatorname{Ric}(\nu,\nu)f^2\right)\,d\mathcal H^n =: Q(f,f).
\]
We say that $M$ is stable if $\delta^2(X)\ge 0$ for every $X\in C_c^1(N,TN)$. The index of $M$, denoted $\operatorname{index}(M)$, is defined as the number of negative eigenvalues of the Jacobi operator
\[
L:=\Delta + |A|^2 + \operatorname{Ric}(\nu,\nu).
\]

If $\operatorname{sing}M\neq \varnothing$, we choose $X\in C_c^1(\operatorname{reg}M)$. In the case where $M$ is two-sided and $\mathcal H^{n-2}(\operatorname{sing}M)=0$, stability is equivalent to the positivity of $Q(f,f)$ for all bounded, locally Lipschitz functions $f$ that may be nonzero over the singularities of $\operatorname{spt}M$ (see page 751 of \cite{SS81} for details).

We next extend the notion of index to general hypersurfaces.

\begin{definition}
Let $M^m\subset N$ be a two-sided $C^2$-hypersurface. For any open set $U\subset N$, the index of $M\cap U$ is the smallest integer $I$ such that every $(I+1)$-dimensional subspace $P\subset C_c^1(M\cap U,TN)$ contains a nonzero vector field $X\in P$ with $\delta^2(X,X)\ge 0$.
\end{definition}

We collect several auxiliary lemmas that will be used throughout the paper.

\begin{lemma}\label{lm1.5}
	There exists no stable stationary varifold $V$ in $\mathbb S^n(1)$ such that $\theta(x)=1$ for $\mathcal H^n$-a.e. $x\in \operatorname{spt}V$ and $\mathcal H^{n-2}(\operatorname{sing} V)=0$.
\end{lemma}
\begin{proof}
	Let $M=\mathrm{supp}\mu_V$, then $V=(M,1)$. Taking $f\equiv 1$ in the stability inequality yields
	\begin{equation*}%\label{eq1.8}
		\int_M \left(n-1+|A|^2\right)\,d\mathcal H^n \le \int_M |\nabla 1|^2\,d\mathcal H^n = 0,
	\end{equation*}
	which is impossible since the integrand is strictly positive. This contradiction proves the lemma.
\end{proof}

\begin{lemma}\cite{Sha17}\label{lm1.7}
	Let $U\subset \mathbb R^{n+1}$ be open, let $I\in \mathbb Z_{\ge 0}$, and let $V$ be a hypersurface with $\operatorname{index}(\operatorname{reg} V)\le I$. If $\Omega_1,\dots,\Omega_{I+1}\subset U$ are $I+1$ open sets such that $\operatorname{index}(\operatorname{reg}V,\Omega_j)\ge 1$ for each $j=1,\dots,I+1$, then there exist $j\ne j'$ with $\Omega_j\cap \Omega_{j'}\neq\varnothing$.
\end{lemma}

%We now recall the notion of convergence for hypersurfaces. For any ball $B_R(p)\subset \mathbb R^{n+1}$, we say that a hypersurface $M^n\subset \mathbb R^{n+1}$ is embedded in $B_R(p)$ if either $M$ is closed or $\partial M\cap \overline B_r(p)=\varnothing$. 
\begin{definition}
We say that a sequence $\{M_i\}$ converges to a hypersurface $M$ (possibly with multiplicity) at $p\in M$ in the $C^{k,\alpha}_{\mathrm{loc}}$ topology if there exists $r>0$ such that for all sufficiently large $i$, $M_i\cap B_r(p)$ is a union of disjoint graphs over the tangent space $T_pM$, and each graph component of $M_i$ converges to the graph of $M$ in the usual $C^{k,\alpha}_{\mathrm{loc}}$ sense. We say that $M_i$ converges to $M$ in $C^{k,\alpha}_{\mathrm{loc}}$ topology if this holds at every $p\in M$. Moreover, $M_i$ is said to converge smoothly to $M$ if it converges in $C^{k,\alpha}_{\mathrm{loc}}$ topology for every $k\in\mathbb Z_{\ge 0}$ and every $0<\alpha<1$. The definition of $C^{k,\alpha}$ convergence extends to general ambient manifolds via the exponential map.
\end{definition}
\begin{lemma}\label{lm5.7}
	Let $\Gamma_i$ and $\Gamma$ be minimal hypersurfaces in $\mathbb S^n$ with $\mathcal H^{n-2}(\operatorname{sing}\Gamma_i)=0$ and $\mathcal H^{n-2}(\operatorname{sing}\Gamma)=0$. Suppose that $\Gamma_i$ converges smoothly to $\Gamma$ (possibly with multiplicity) away from $\operatorname{sing}\Gamma$. Then $\Gamma_i$ converges smoothly to $\Gamma$ with multiplicity one away from $\operatorname{sing}\Gamma$.
\end{lemma}

\begin{proof}
	Assume, toward a contradiction, that there exist $\Gamma_i$, $\Gamma$, and an integer $m>1$ such that $\Gamma_i\rightharpoonup m\Gamma$ in the varifold sense. Then one can construct a positive, nowhere-vanishing Jacobi field $f>0$ on $\operatorname{reg}\Gamma$, which implies that $\Gamma$ is stable in $\mathbb S^n$ (cf.\ \cite{Sha17}). However, Lemma~\ref{lm1.5} asserts that no stable minimal hypersurface exists in $\mathbb S^n$, yielding the desired contradiction.
\end{proof}

\begin{lemma}\label{index-bound}
	Let $M_k$ and $M$ be $C^2$-hypersurfaces in $N$ such that $\mathcal H^{n-2}(\operatorname{sing}M_k)=\mathcal H^{n-2}(\operatorname{sing}M)=0$. Suppose that $M_k\rightharpoonup mM$ in the varifold sense for some $m\in\mathbb Z_{\ge 0}$, and that the convergence is $C^2$. Then for any open set $U\subset N$, we have
    $$\operatorname{index}(M\cap U)\le \liminf_{k\to\infty}\operatorname{index}(M_k\cap U).$$
	In particular, if each $M_k$ is stable in $U$, then $M$ is stable in $U$.
\end{lemma}
\begin{proof}
	If not, then $\operatorname{index}(M)>I$. Equivalently, there exist $I+1$ functions $\{f_i\}_{i=1}^{I+1}$, of unit norm and mutually orthogonal in $L^2$, satisfying
\begin{equation*}
    Q(f_i,f_i)<0 \qquad \forall i=1,\dots,I+1.
\end{equation*}
    Let $\tilde{f}_i$ be a $C^1$ extension of  $f_i$ to $\mathbb{R}^{n+1}$ and set $f^k_i=\tilde{f}_i|_{M_k}$. By the $C^2$ convergence, we know that 
    \begin{equation*}
        \lim\limits_{k\to \infty}Q(f^k_i,f^k_i)= \lim\limits_{k\to \infty}\int_{M_k}\left(|\nabla f^k_i|^2 - |A|^2 (f_i^k)^2  \right)\,d\mathcal H^n=Q(f_i,f_i)<0,
    \end{equation*}
    for all $i$. Then $f^k_i$ are linearly  dependent for  sufficiently large $k$, otherwise $\mathrm{index}(M_k)>I$.
Then there  are $\lambda^k_1,\dots,\lambda^k_{I+1}\in \mathbb{R}$, not all zero, such that $\lambda^k_1f^k_1+\dots+\lambda^k_{I+1}f^k_{I+1}=0$  for  $k$ large enough. We assume that $\max_{i}|\lambda^k_i|=1$.  On the other hand, by direct calculation, we have
\begin{equation*}
    \lim\limits_{k\to \infty}\int_{M_k} f^k_if^k_j=m\int_{M} f_if_j=\delta_{ij}.
\end{equation*}
Hence,
\begin{equation*}
        0=\lim\limits_{k\to \infty}\int_{M_k}(\lambda^k_1f^k_1+\dots+\lambda^k_{I+1}f^k_{I+1})^2=m\lim\limits_{k\to \infty}\sum_{i=1}^{I+1}(\lambda^k_j)^2\geq m,
\end{equation*}
a contradiction, which completes the proof.
\end{proof}

\begin{lemma}\label{stable-cone}
	Let $\mathbf C\subset \mathbb R^{n+1}$ be a minimal cone with $\mathcal H^{n-2}(\operatorname{sing}\mathbf C)=0$. If $\mathbf C$ has finite index, then $\mathbf C$ is stable and $\mathcal H^{n-7+\alpha}(\operatorname{sing}\mathbf C)=0$ for every $\alpha>0$.
\end{lemma}
\begin{proof}
	Suppose, to the contrary, that $\mathbf C$ is not stable. Then there exists $\phi\in C_c^1(\operatorname{reg}\mathbf C)$ such that $Q(\phi,\phi)<0$. Choose $R>1$ such that $\operatorname{supp}\phi\subset (B_R(0)\setminus B_{R^{-1}}(0))\cap \mathbf C$. For each $k\in\mathbb Z_{\ge 0}$, define
	\[
		\phi_k(x):=R^{2k}\phi(R^{-2k}x).
	\]
	Then $\operatorname{supp}\phi_k\subset (B_{R^{2k+1}}(0)\setminus B_{R^{2k-1}}(0))\cap \mathbf C$, and hence the supports of $\phi_i$ and $\phi_j$ are disjoint for $i\ne j$. Moreover, by the conical structure of $\mathbf C$,
    $Q(\phi_k,\phi_k)=R^{2nk}Q(\phi,\phi)<0.$
	Thus the support of $\phi_1,\dots,\phi_{I+1}$ provide $I+1$ disjoint regions each with negative index contribution, which contradicts Lemma~\ref{lm1.7}. Hence $\mathbf C$ must be stable. The Hausdorff dimension estimate follows from the standard stratification theorem for stable minimal cones (see, e.g., \cite{Sha17}).
\end{proof}

\subsection{F-stationary rectifiable varifolds}$\\$

We recall the variational characterization of shrinkers as Gaussian minimal submanifolds. Following \cite{CM12}, the Gaussian area of an $n$-dimensional rectifiable varifold $\Sigma^n\subset\mathbb R^N$ is defined by
\begin{equation*}
	F(\Sigma):=(4\pi)^{-\frac{n}{2}}\int_\Sigma e^{-\frac{|x|^2}{4}}\,d\mathcal H^n.
\end{equation*}
Critical points of $F$ are called $F$-stationary. If $\Sigma$ is an immersed submanifold, then $F$-stationarity is equivalent to the shrinker equation
\begin{equation*}
	\mathbf H+\frac{x^\perp}{2}=0,
\end{equation*}
where $\mathbf H$ denotes the mean curvature vector and $x^\perp$ is the normal component of the position vector field. In particular, any stationary cone with vertex at the origin is $F$-stationary, since in that case $\mathbf H=0$ and $x^\perp=0$.

\begin{definition}\label{L-stab}
	Let $V$ be an $F$-stationary integral $n$-varifold in $\mathbb R^{n+1}$. For $R>0$, we say that $V$ is $L$-stable in the ball $B_R(0)$ if for every function $\varphi\in W^{1,2}_0(B_R(0)\setminus \operatorname{sing}V)$,
	\begin{equation}\label{eq5.11}
		-\int \varphi\, L\varphi\, e^{-\frac{|x|^2}{4}}\,d\mu_V \ge 0,
	\end{equation}
	where $L$ is the second-order operator given by
	\begin{equation}\label{eq5.12}
		L:=\Delta_V -\frac{1}{2}\langle x,\nabla(\cdot)\rangle + |A|^2 + \frac{1}{2}.
	\end{equation}
	We say that $V$ is not $L$-stable in $B_R(0)$ if there exists some $\varphi\in W^{1,2}_0(B_R(0)\setminus \operatorname{sing}V)$ for which \eqref{eq5.11} fails. Finally, $V$ is called $L$-stable in $\mathbb R^{n+1}$ if it is $L$-stable in $B_R(0)$ for every $R>0$; otherwise, it is called $L$-unstable.
\end{definition}

Recently, Colding and Minicozzi proved that every $F$-stationary varifold is $L$-unstable.

\begin{lemma}\cite{CM23}\label{lm-b}
	Let $\Sigma\subset B_R(0)\subset \mathbb R^N$ be a proper, $F$-stationary, rectifiable varifold with $R\ge R_n$ and $\mathcal H^{n-2}(\operatorname{sing}\Sigma)=0$, where $R_n$ is a constant depending only on $n$. Then $\Sigma$ is $L$-unstable in $B_{R}(0)$.
\end{lemma}
\subsection{Mean curvature flow}$\\$

We recall the standard volume ratio bound for the mean curvature flow(cf. ~\cite{CM12}).

\begin{lemma}\label{vol}
	Let $\{(M^n,\mathbf{x}(t)):0\le t<T\}$ be a mean curvature flow. Then there exists a constant $N_0>0$ such that for every $r>0$ and every $p_0\in\mathbb R^{n+1}$,
	\begin{equation*}
		\mathrm{Vol}(B_r(p_0)\cap M_t)\le N_0 r^n,\qquad t\in[0,T).
	\end{equation*}
    The constant $N_0$ depends on $T$ and the volume ratio of $M_0$.
\end{lemma}

To study the convergence of mean curvature flows, we introduce the following notion of convergence for sequences of one-parameter hypersurfaces.

\begin{definition}
	Let $\{M^n_{i,t}:-1<t<1\}$ be a sequence of one-parameter smoothly embedded hypersurfaces. We say that this sequence converges in the smooth topology, possibly with multiplicities, to a limit flow $\{M^n_{\infty,t}:-1<t<1\}$ away from a space-time singular set $\mathcal S\subset\mathbb R^{n+1}\times(-1,1)$, if for every $t\in(-1,1)$, every $p\in M_{\infty,t}\setminus\mathcal S_t$, and all sufficiently large $i$, there exists $\varepsilon>0$ such that $M_{i,s}\cap B_\varepsilon(p)$ for $s\in[t-\varepsilon^2,t+\varepsilon^2]$ can be represented as a collection of graphs
	\[
		\{u_i^1(x,s),\dots,u_i^N(x,s)\}
	\]
	over $M_{\infty,s}$ for some positive integer $N$. Moreover, for each $k=1,\dots,N$, the functions $u_i^k(x,s)$ converge smoothly to zero locally in both $x$ and $s$ as $i\to\infty$. Here  $\mathcal S_t:=\{x\in\mathbb R^{n+1}:(x,t)\in\mathcal S\}$.
\end{definition}

We recall the following classical compactness result for the mean curvature flow.

\begin{theorem}[{\cite{CY07}}]\label{mcf-cov}
	Let $\{(M^n_i,\mathbf{x}_i(t)):-1<t<1\}$ be a sequence of mean curvature flows properly immersed in $B_R(0)\subset\mathbb R^{n+1}$. Suppose that for some $\Lambda>0$,
	\begin{equation*}
		\sup_{M_{i,t}\cap B_R(0)}|A|(x,t)\le \Lambda,\qquad t\in(-1,1).
	\end{equation*}
	Then a subsequence of $\{(M_{i,t},\mathbf{x}_i(t)):-1<t<1\}$ converges in the smooth topology to a smooth mean curvature flow $\{(M_{\infty,t},\mathbf{x}(t)):-1<t<1\}$ in $B_R(0)$.
\end{theorem}

\section{Weak compactness of mean curvature flow}\label{sec3}

	In this section, we study the weak compactness of the
mean curvature flow under some geometric conditions. This weak compactness
result will be used to prove the convergence of the rescaled mean curvature flow
in the next section.

\begin{definition}\label{def4.1}
	Let $\{(\Sigma^n_i,\mathbf{x}_i(t)),-1\leq t\leq 1\}$ be a one parameter family of closed smooth embedded hypersurfaces satisfying the mean
	curvature flow equation. It is called a refined sequence if the following
	properties are satisfied for every $i$ :
	\begin{itemize}
		\item [(1)] There exists a constant $D>0$ such that $\mathrm{d}(\Sigma_{i,t},0)\leq D$, where $d(\Sigma,0)$
		denotes the Euclidean distance from the point
		$0\in\mathbb{R}^{n+1}$ to the surface $\Sigma\subset \mathbb{R}^{n+1}$.
		\item [(2)] The mean curvature satisfies the following condition
		\begin{equation}\label{eq4.1}
			\lim\limits_{i\to \infty}\sup\limits_{\Sigma_{i,t}\times [-1,1]}|H_i|(p,t)=0.
		\end{equation}
		\item [(3)]  There is a uniform constant $I$ such that
		\begin{equation}
			\mathrm{index}(\Sigma_{i,0})\leq I.
		\end{equation}
		\item [(4)] There is uniform $N > 0$ such that for all $r > 0$ and $p\in \mathbb{R}^{n+1}$, we have
		\begin{equation}
			\begin{split}
				r^{-n}\mathrm{Vol}(B_r(p)\cap \Sigma_{i,t})\leq N\quad \forall t\in [-1,1]
			\end{split}
		\end{equation}
		\item [(5)] There exist uniform constants $\bar{r},\kappa> 0$ such that for any $r \in (0,\bar{r}]$ and any
		$p\in \Sigma_{i,t}$ we have
		\begin{equation}
			r^{-n}\mathrm{Vol}(B_r(p)\cap \Sigma_{i,t})\geq \kappa \quad \forall t\in[-1,1].
		\end{equation}
	\end{itemize}
\end{definition}
\begin{theorem}[Weak compactness of refined sequences]\label{prop4.2}
	If $\{(\Sigma^n_i,\mathbf{x}_i(t)),-1\leq t\leq 1\}$ is a refined sequence in the sense of Definition \ref{def4.1} then there
	exists a finite set of points $\mathcal{S}'\subset \mathbb{R}^{n+1}$ and a minimal hypersurface $\Sigma_{\infty}$ with index at most $I$ such that a subsequence of $\{(\Sigma^n_i,\mathbf{x}_i(t)),-1\leq t\leq 1\}$ converges in the smooth
	topology, possibly with multiplicity at most $N$ to $\Sigma_{\infty}$ away from $\mathcal{S}'\cup \mathrm{sing}\Sigma_\infty$. Here $\CS'$ has at most $I$ points.  Furthermore, the
	subsequence also converges to $\Sigma_{\infty}$ in the $($extrinsic$)$ Hausdorff distance.
\end{theorem}

\begin{remark}
    The singular convergence can be divided into two cases. 
In the first case, the flow converges to a singular point; 
in the second, it converges to a regular part, but the convergence is not smooth. 
Locally, the flow consists of planes connected by necks. 
Upon rescaling, the flow may converge to a minimal hypersurface with flat ends (see Figure \ref{fig2}).

\end{remark}

\begin{figure}[htbp]
    \centering
   \resizebox{0.4\textwidth}{!}{

\tikzset{every picture/.style={line width=0.75pt}} %set default line width to 0.75pt        

\begin{tikzpicture}[x=0.75pt,y=0.75pt,yscale=-1,xscale=1]
%uncomment if require: \path (0,441); %set diagram left start at 0, and has height of 441

%Straight Lines [id:da021976687612252577] 
\draw    (170.89,93.92) -- (324.23,362.53) ;
%Straight Lines [id:da027863291427658976] 
\draw    (324.23,362.53) -- (482.23,93.53) ;
%Shape: Arc [id:dp3481057786447693] 
\draw  [draw opacity=0] (482.23,93.53) .. controls (468.35,107.42) and (403.84,117.92) .. (326.36,117.92) .. controls (249.61,117.92) and (185.59,107.62) .. (170.89,93.92) -- (326.36,87.92) -- cycle ; \draw   (482.23,93.53) .. controls (468.35,107.42) and (403.84,117.92) .. (326.36,117.92) .. controls (249.61,117.92) and (185.59,107.62) .. (170.89,93.92) ;  
%Shape: Arc [id:dp176615909407133] 
\draw  [draw opacity=0][dash pattern={on 4.5pt off 4.5pt}] (170.92,94.01) .. controls (170.85,93.65) and (170.82,93.29) .. (170.82,92.92) .. controls (170.82,76.35) and (240.53,62.92) .. (326.52,62.92) .. controls (412.52,62.92) and (482.23,76.35) .. (482.23,92.92) .. controls (482.23,93.28) and (482.2,93.63) .. (482.14,93.98) -- (326.52,92.92) -- cycle ; \draw  [dash pattern={on 4.5pt off 4.5pt}] (170.92,94.01) .. controls (170.85,93.65) and (170.82,93.29) .. (170.82,92.92) .. controls (170.82,76.35) and (240.53,62.92) .. (326.52,62.92) .. controls (412.52,62.92) and (482.23,76.35) .. (482.23,92.92) .. controls (482.23,93.28) and (482.2,93.63) .. (482.14,93.98) ;  
%Curve Lines [id:da3911566187867911] 
\draw    (147,114) .. controls (187,84) and (291,399) .. (331,369) ;
%Curve Lines [id:da3022179396910377] 
\draw    (331,369) .. controls (371,339) and (346,308) .. (386,278) ;
%Curve Lines [id:da27038473868290713] 
\draw    (172,87) .. controls (212,57) and (288,364) .. (328,334) ;
%Curve Lines [id:da9455597463450376] 
\draw    (328,334) .. controls (368,304) and (330.9,298.72) .. (370.9,268.72) ;
%Shape: Arc [id:dp10546302670055097] 
\draw  [draw opacity=0] (370.9,268.72) .. controls (371.62,267.98) and (372.35,267.23) .. (373.1,266.49) .. controls (388.56,251.02) and (403.83,241.22) .. (407.2,244.59) .. controls (410.52,247.91) and (401.06,262.78) .. (386,278) -- (379.2,272.59) -- cycle ; \draw   (370.9,268.72) .. controls (371.62,267.98) and (372.35,267.23) .. (373.1,266.49) .. controls (388.56,251.02) and (403.83,241.22) .. (407.2,244.59) .. controls (410.52,247.91) and (401.06,262.78) .. (386,278) ;  
%Shape: Arc [id:dp4198394939405372] 
\draw  [draw opacity=0] (455,172.61) .. controls (454.51,173.49) and (454.01,174.37) .. (453.5,175.25) .. controls (438.86,200.6) and (424.09,219.47) .. (420.5,217.4) .. controls (416.91,215.33) and (425.87,193.1) .. (440.5,167.75) .. controls (440.96,166.96) and (441.42,166.17) .. (441.88,165.39) -- (447,171.5) -- cycle ; \draw   (455,172.61) .. controls (454.51,173.49) and (454.01,174.37) .. (453.5,175.25) .. controls (438.86,200.6) and (424.09,219.47) .. (420.5,217.4) .. controls (416.91,215.33) and (425.87,193.1) .. (440.5,167.75) .. controls (440.96,166.96) and (441.42,166.17) .. (441.88,165.39) ;  
%Curve Lines [id:da5483659958737843] 
\draw    (455,172.61) .. controls (506,91) and (457,157) .. (506,89) ;
%Curve Lines [id:da848253069214415] 
\draw    (441.88,165.39) .. controls (471.88,123.39) and (442.23,123.53) .. (482.23,93.53) ;
%Shape: Arc [id:dp030921900494636945] 
\draw  [draw opacity=0] (421.9,216.69) .. controls (423.98,221.19) and (422.94,227.96) .. (418.83,233.83) .. controls (414.83,239.54) and (409.05,242.81) .. (404.19,242.6) -- (409.33,227.18) -- cycle ; \draw   (421.9,216.69) .. controls (423.98,221.19) and (422.94,227.96) .. (418.83,233.83) .. controls (414.83,239.54) and (409.05,242.81) .. (404.19,242.6) ;  
%Shape: Arc [id:dp5652237982483376] 
\draw  [draw opacity=0][dash pattern={on 0.84pt off 2.51pt}] (404.19,242.6) .. controls (402.87,232.85) and (406.27,223.34) .. (413.98,217.94) .. controls (415.57,216.82) and (417.26,215.94) .. (419.02,215.28) -- (431.19,242.51) -- cycle ; \draw  [dash pattern={on 0.84pt off 2.51pt}] (404.19,242.6) .. controls (402.87,232.85) and (406.27,223.34) .. (413.98,217.94) .. controls (415.57,216.82) and (417.26,215.94) .. (419.02,215.28) ;  
%Straight Lines [id:da33149958832364534] 
\draw    (325,54) -- (325,102) ;
\draw [shift={(325,104)}, rotate = 270] [color={rgb, 255:red, 0; green, 0; blue, 0 }  ][line width=0.75]    (10.93,-3.29) .. controls (6.95,-1.4) and (3.31,-0.3) .. (0,0) .. controls (3.31,0.3) and (6.95,1.4) .. (10.93,3.29)   ;
%Straight Lines [id:da8555669684016358] 
\draw    (477,229) -- (431,229.96) ;
\draw [shift={(429,230)}, rotate = 358.81] [color={rgb, 255:red, 0; green, 0; blue, 0 }  ][line width=0.75]    (10.93,-3.29) .. controls (6.95,-1.4) and (3.31,-0.3) .. (0,0) .. controls (3.31,0.3) and (6.95,1.4) .. (10.93,3.29)   ;
%Straight Lines [id:da8315783921805918] 
\draw    (231.23,388.53) -- (322.31,363.07) ;
\draw [shift={(324.23,362.53)}, rotate = 164.38] [color={rgb, 255:red, 0; green, 0; blue, 0 }  ][line width=0.75]    (10.93,-3.29) .. controls (6.95,-1.4) and (3.31,-0.3) .. (0,0) .. controls (3.31,0.3) and (6.95,1.4) .. (10.93,3.29)   ;

% Text Node
\draw (253,7.4) node [anchor=north west][inner sep=0.75pt]  [font=\Huge]  {$stable\ cone$};
% Text Node
\draw (431,233.4) node [anchor=north west][inner sep=0.75pt]  [font=\Huge]  {$neck\ region$};
% Text Node
\draw (39,331.4) node [anchor=north west][inner sep=0.75pt]  [font=\Huge]  {$singular\ point$};

\end{tikzpicture}

}
    \caption{Weak convergence}
    \label{fig2}
\end{figure}

To prove this theorem, we follow the method of Li and Wang \cite{LW19}. We need to establish a long-time pseudolocality theorem, which plays a crucial role in their proof. Our proof of the pseudolocality theorem differs from theirs.

In dimension two, using the Gauss–Bonnet formula, they obtained a local estimate for $\int |A|^2$
 . Then, via a point-picking trick, they derived a pointwise estimate for the second fundamental form. In higher dimensions, we lack such a curvature integral estimate.

The bounded mean curvature assumption allows us to apply Allard's regularity theorem to represent the flow as a family of graphs; then, using the interior estimate of Ecker and Huisken \cite{EH91}, we obtain a pointwise estimate. To that end, we need to show that the volume ratio is close to one. To overcome this difficulty, we assume that the flow has uniformly bounded Morse index and invoke the classical work of Schoen and Simon \cite{SS81}.

\subsection{Two-sided pseudolocality theorem}$\\$

\begin{lemma}\label{EH}\cite[Theorem 3.1]{EH91}
	Let $R>0$ be such that $\{x\in M_t| r(x,t)\leq R^2\}$ is compact and can be written as a graph over some hyperplane for $t\in [0,T]$. Then for any  $t\in [0,T]$ and $0\leq \xi<1$ we have the estimate 
	\begin{equation*}
		\sup\limits_{\{x\in M_t| r(x,t)\leq \xi R^2\}}|A|^2\leq c(n)(1-\xi )^{-2}(t^{-1}+R^{-2})\sup\limits_{\{x\in M_s| r(x,s)\leq R^2,s\in [0,t]\}}\nu^4,
	\end{equation*}
	where $r(x,t)=|x|^2+2nt$ and $\nu=\langle \mathbf{n},v_0 \rangle^{-1}$, $v_0$ is the unit normal vector of the plane.
\end{lemma}
\begin{definition}
	For any $r>0$, $p\in\mathbb{R}^{n+m}$ and $\Sigma^n\subset \mathbb{R}^{n+m}$ we denote by $C_x(B_r(p)\cap \Sigma)$ the connected component of $B_r(p)\cap \Sigma$ containing $x\in\Sigma$.
\end{definition}
\begin{lemma}\cite[Lemma 3.4]{LW19}\label{lm3}
	Let $\{(M^n,\mathbf{x}(t)),-T\leq t\leq T\}$ be a mean curvature flow in $\mathbb{R}^{n+1}$ with $\max_{M_t\times [-T,T]}|H|(p,t)\leq \Lambda$. Then for any $t_1,t_2\in [-T,T]$ and $0<r_1<r_2<r_0$, we have 
	\begin{equation*}
		\begin{split}
			\frac{\mathrm{Vol}(C_{p_{t_2}}(B_{r_2}(p_{t_2})\cap M_{t_2}))}{\omega_n r_2^n}	\leq f(t_1,t_2,\Lambda, r_2)\frac{\mathrm{Vol}(C_{p_{t_1}}(B_{r_1}(p_{t_1})\cap M_{t_1}))}{\omega_n r_1^n}
		\end{split}
	\end{equation*}
	where $p_t=\mathbf{x}_t(p)$ for some $p\in M$  and 
	\begin{equation*}
		\begin{split}
			r_1&=r_2+2\Lambda |t_2-t_1|,\\
			f(t_1,t_2,\Lambda, r_2)&=e^{\Lambda^2|t_2-t_1|}\left(1+\frac{2\Lambda}{r_2}|t_2-t_1|\right)^n.
		\end{split}
	\end{equation*}
\end{lemma}

\begin{lemma}\cite[Lemma 3.5]{LW19}\label{lm4}
	Let $\Sigma^n\subset \mathbb{R}^{n+1}$ be a properly embedded hypersurface in $B_{r_0}(x_0)$ with $x_0\in \Sigma$ and $|H|\leq \Lambda$. Then for any $s\in (0,r_0)$ we have
	\begin{equation*}
		\frac{\mathrm{Vol}(B_s(x_0)\cap  \Sigma)}{\omega_n s^n}\leq e^{\Lambda r_0} \frac{\mathrm{Vol}(B_{r_0}(x_0)\cap  \Sigma)}{\omega_n r_0^n}
	\end{equation*}
	In particular, letting $s\to  0$ we have 
	\begin{equation*}
		\mathrm{Vol}(B_r(x_0)\cap  \Sigma)\geq e^{-\Lambda r}\omega_n r^n\quad \forall\; r\in (0,r_0].
	\end{equation*}
\end{lemma}

\begin{theorem}\label{pseudo1}
	For any $r_0\in (0,1]$ and $T>1$, there exist $\delta=\delta(n,r_0,T),\varepsilon=\varepsilon(n)>0,\theta=\theta(n)>0$ with the following properties. Let $\{(M^n,\mathbf{x}(t)),-T\leq t\leq T\}$ be a closed smooth embedded mean curvature flow. Assume that
	\begin{itemize}
		\item [(1)] $\mathrm{Vol}(C_{p_0}( B_{r_0}(p_0)\cap M_0))\leq (1+\theta)\omega_n r_0^n$ where $p_0=\mathbf{x}_0(p)$ for some $p\in M$.
		\item [(2)] the mean curvature of $\{(M^n,\mathbf{x}(t)),-T\leq t\leq T\}$ is bounded by $\delta$.
	\end{itemize}
	Then for any $(x,t)\in C_{p_t}(M_t\cap B_{\varepsilon r_0}(p_0))\times [-T,T]$ 
	where $p_t=\mathbf{x}_t(p)$, we have the estimate
	\begin{equation*}
		|A|(x,t)\leq \frac{1}{\varepsilon r_0}.
	\end{equation*}
\end{theorem}
\begin{proof}
$\\$
\textbf{Step 1. 	The volume ratio is close to one for any $t\in [-T,T]$.}

	Without loss of generality, we assume $r_0 = 1$. Let $p=n+1$ in Lemma \ref{allard}, we choose $2\theta<\lambda_0$ such that $C(2\theta)^{\frac{1}{2n+2}
	}<\frac{1}{128}$. Noting that 
	\begin{equation*}
		|p_t-p_0|=\bigg|\int_0^t H(p,s)\mathbf{n}(p,s)ds\bigg|\leq \delta T,
	\end{equation*}
	we can choose  $\delta$ small enough such that  $|p_0-p_t|\leq \gamma/16$ for $t\in [-T,T]$, where $\gamma$ is defined in Lemma \ref{allard}.
	
	Using Lemma \ref{lm3} and Lemma \ref{lm4}, we know that when $\delta $ is small enough we have 
	\begin{equation*}
		\begin{split}
			\frac{\mathrm{Vol}_{g(t)}(C_{p_{t}}(B_{1/2}(p_t)\cap M_{t}))}{\omega_n (1/2)^n}	&\leq e^{\delta r_t} \frac{\mathrm{Vol}_{g(t)}(C_{p_t}(B_{r_t}(p_t)\cap M_{t}))}{\omega_n r_t^n}\\
			&\leq e^{\delta r_t} f(0,t,\delta,r_t)\frac{\mathrm{Vol}_{g(0)}(C_{p_{0}}(B_{1}(p_{0})\cap M_{0}))}{\omega_n 1^n}\\
			&\leq 1+2\theta,
		\end{split}
	\end{equation*}
	where $1=r_t+2\delta|t|$.

    \textbf{Step 2. Each time-slice can be locally written as a graph.}
    
    By direct calculation, we have
	\begin{equation*}
		\begin{split}
			\bigg(2^{n-p}\int_{B_{\frac12}(p_t)}|H|^pd\mu_{M_t}\bigg)^{1/p}\leq \delta \bigg(2^{n-p}\int_{B_{\frac12}(p_t)}1d\mu_{M_t}\bigg)^{1/p}\leq   2\theta.
		\end{split}
	\end{equation*}
	Then by Allard's regularity theorem, for any $t\in [-T,T]$,  there exist an orthogonal transformation $Q_t$ of $\mathbb{R}^{n+1}$ and a $\Tilde{u}_t\in C^{1,\frac{1}{n+1}}$ with $D\Tilde{u}_t(0)=0$, $C_{p_t}(M_t\cap B_{\gamma/2}(p_t))=Q_t$(graph $\Tilde{u}_t$$\cap B_{\gamma/2}(0))$, and 
	\begin{equation}\label{eq2.11}
		\begin{split}
			\sup\limits_{x\in B_{\gamma/2}(p_t)} |D\Tilde{u}(x)|+  \sup\limits_{x,y\in B_{\gamma/2}(p_t),x\neq y}{|D\Tilde{u}_t(x)-D\Tilde{u}_t(y)|}\leq C(2\theta)^{\frac{1}{2n+2}}\leq \frac{1}{128}.
		\end{split}
	\end{equation}

	We claim that if $\delta$ is small enough then there is $u(x,t)$ such that
	$C_{p_t}(M_t\cap B_{\gamma/4}(p_0))=Q$(graph $u(\cdot,t)$$\cap B_{\gamma/4}(0))$ for any $t\in [-T,T]$. In fact, for any  $\phi\in C^1_{c}(B_{\gamma/4}(p_0))$, noting that $\partial_t \mathbf{n}=\nabla^{M_t} H$(cf.\cite{Hui84}), we have 
	\begin{equation}\label{eq2.12}
		\begin{split}
			\frac{d}{dt}\int_{M_{t}} \phi \mathbf{n}(x,t)&=\int_{M_t}\phi\nabla^{M_t}  H-\int_{M_t}H\nabla^\perp \phi -\int_{M_t} \phi|H|^2\mathbf{n}\\
			&=\int_{M_t}\nabla^{M_t}(\phi H)-\int_{M_t}H\nabla^{M_t}\phi -\int_{M_t}H\nabla^\perp \phi -\int_{M_t}\phi|H|^2\mathbf{n}\\
			&=\int_{M_t}\nabla^{M_t}(\phi H)-\int_{M_t}H\nabla\phi -\int_{M_t}\phi|H|^2\mathbf{n}.
		\end{split}
	\end{equation}
	Choose a standard orthonormal basis $e_1,\cdots,e_{n+1}$ of $\IR^{n+1}$. By direct calculation, we have
	\begin{equation}\label{v2-2.12}
		\begin{split}
			\int_{M_t}\nabla^{M_t}(\phi H)&=\sum_{i=1}^{n+1}\int_{M_t}\langle \nabla^{M_t}(\phi H),e_i\rangle dV_{M_t}e_i\\
			&=\sum_{i=1}^{n+1}\int_{M_t}\text{div}(\phi H e_i)dV_{M_t}e_i-\int_{M_t}\phi H \text{div}(e_i)dV_{M_t}e_i\\
			&=\sum_{i=1}^{n+1}\int_{M_t}\langle \phi H e_i, H\mathbf{n}\rangle dV_{M_t}e_i\\
			&=\int_{M_t}\phi H^2 \mathbf{n}
		\end{split}
	\end{equation}
	and
	\begin{equation}\label{v2-2.13}
		\begin{split}
			\frac{d}{dt}\int_{M_{t}} \phi=-\int_{M_{t}}H^2 \phi-\int_{M_{t}}H\nabla \phi\cdot\mathbf{n}.
		\end{split}
	\end{equation}
	Choose $\phi\geq 0$ so that $\phi\equiv1$ for  $d(x,p_0)\leq \gamma/8$ and $\phi\equiv 0$ for $d(x,p_0)\geq \gamma/4$ and $|\nabla \phi|\leq 16\gamma^{-1}$. Using \eqref{eq2.12}, \eqref{v2-2.12} and \eqref{v2-2.13},  we have
	\begin{equation}\label{eq2.13}
		\begin{split}
			\bigg|\frac{d}{dt}\frac{\int_{M_{t}} \phi \mathbf{n}(x,t)}{\int_{M_{t}} \phi }\bigg|&\leq  \bigg|\frac{\frac{d}{dt}\int_{M_{t}} \phi \mathbf{n}(x,t)}{\int_{M_{t}} \phi }\bigg|+\bigg|\frac{\int_{M_{t}} \phi \mathbf{n}(x,t)\frac{d}{dt}\int_{M_{t}} \phi}{(\int_{M_{t}} \phi )^2}\bigg|\\
			&\leq \frac{|\int_{M_t}|\nabla\phi| |H|}{\int_{M_{t}} \phi }+\frac{\int_{M_{t}} \phi(\int_{M_{t}}H^2 \phi+|\int_{M_{t}}|H| |\nabla\phi|)}{(\int_{M_{t}} \phi )^2}\\
			&\leq \frac{2\delta\int_{M_t}|\nabla\phi| +\delta^2\int_{M_t}\phi}{\int_{M_{t}} \phi }\\
			&\leq \frac{1}{32T} \\
		\end{split}
	\end{equation}
	for $t\in [-T,T]$, if  $\delta $ is small enough.
	
	Combining \eqref{eq2.11} and \eqref{eq2.13}, we obtain
	\begin{equation}\label{eq2.15}
		\begin{split}
			&\bigg|\mathbf{n}(p,0)-\mathbf{n}(p,t)\bigg|\\
			=&\bigg|\frac{\int_{M_{0}} \phi \mathbf{n}(p,0)}{\int_{M_0} \phi }-\frac{\int_{M_t} \phi \mathbf{n}(p,t)}{\int_{M_t} \phi }\bigg|\\
			\leq   &\bigg|\frac{\int_{M_0} \phi \mathbf{n}(p,0)}{\int_{M_0} \phi }-\frac{\int_{M_{0}} \phi \mathbf{n}(x,0)}{\int_{M_{0}} \phi }\bigg|+\bigg|\frac{\int_{M_{t}} \phi \mathbf{n}(p,t)}{\int_{M_{t}} \phi }-\frac{\int_{M_{t}} \phi \mathbf{n}(x,t)}{\int_{M_{t}} \phi }\bigg|\\
			&\quad +\bigg|\int_{0}^{t}\frac{d}{ds}\frac{\int_{M_{s}} \phi \mathbf{n}(x,s)}{\int_{M_{s}} \phi }\bigg| \leq \frac{1}{8}.
		\end{split}
	\end{equation}
	Combining \eqref{eq2.11} and \eqref{eq2.15}, we have $|\mathbf{n}(p,0)-\mathbf{n}(x,t)|\leq \frac{1}{4}$ for $x\in B_{\gamma/4}(p_0)\cap  M_t,t\in [-T,T]$. The claim follows form choosing $Q$ to be the tangent plane of $M_0$ at $p_0$.

    \textbf{ Step 3. Point-wise estimate for the second fundamental form.}
    
	By Lemma \ref{EH}, there exit $\varepsilon=\varepsilon(n)>0$ such that $|A|(x,t)\leq \varepsilon$ for any $x\in B_{\gamma/8}(p_0)\cap  M_t,t\in [-T+1,T]$. Choose $\gamma_0=\gamma/8$.
    
    \textbf{ Step 4. Extend the time interval.}

    By Lemma \cite[Theorem 3.8]{LW19}, there exit $\varepsilon=\varepsilon(n)>0$ such that $|A|(x,t)\leq \varepsilon$ for any $x\in B_{\gamma/8}(p_0)\cap  M_t,t\in [-T+1,T]$. Choose $\gamma_0=\gamma/8$.
    
    We complete the proof.
\end{proof}
\begin{remark}
    To prove the main theorem, it suffices to establish the curvature estimate for $t \in [-T+1, T]$.
\end{remark}

Similarly, we also have the following  two-sided pseudolocality theorem when $H$ is uniformly bounded rather than small enough. We omit the proof here.
\begin{theorem}\label{pseudo2}
For any $r_0\in (0,1],T>1$ and $\Lambda>0$, there exist $\delta=\delta(n,T,r_0),\varepsilon=\varepsilon(\Lambda,n)>0,\theta=\theta(\Lambda,n)>0,\eta=\eta(\Lambda,n)>0$ satisfying
\begin{equation*}
	\begin{split}
		\lim\limits_{\Lambda\to 0}\eta(n,\Lambda)=\eta_0(n)>0,\; \lim\limits_{\Lambda\to 0}\varepsilon(n,\Lambda)=\varepsilon_0(n)>0,\; \lim\limits_{\Lambda\to 0}\theta(n,\Lambda)=\theta_0(n)>0
	\end{split}
\end{equation*}
 and the following properties. Let $\{(M^n,\mathbf{x}(t)),-T\leq t\leq T\}$ be a closed smooth embedded mean curvature flow. Assume that
\begin{itemize}
	\item [(1)] $C_{p_0}(M_0\cap B_{r_0}(p_0))\leq (1+\theta)\omega_n r_0^n$ where $p_0=\mathbf{x}_0(p)$ for some $p\in M$.
	\item [(2)] the mean curvature of $\{(M^n,\mathbf{x}(t)),-T\leq t\leq T\}$ is bounded by $\Lambda$.
\end{itemize}
Then for any $(x,t)$ satisfying

$$x\in C_{p_t}(M_t\cap B_{\varepsilon r_0}(p_0)),t\in\bigg[-\frac{\eta r_0^2}{4(\Lambda +\Lambda^2)},\frac{\eta r_0^2}{4(\Lambda +\Lambda^2)}\bigg] \cap [-T,T]$$

where $p_t=\mathbf{x}_t(p)$, we have the estimate
\begin{equation}
	|A|(x,t)\leq \frac{1}{\varepsilon r_0}.
\end{equation}
\end{theorem}

\subsection{Weak compactness of hypersurface}$\\$

In \cite{SS81}, Schoen-Simon proved that
	\begin{theorem}\label{thm3.2}
		Suppose $\{M_k\}$ is a sequence of orientable $C^2$ minimal hypersurfaces with
		\begin{equation}
			\begin{split}
				0\in \bar{M}_k,\mathcal{H}^{n-2}(\mathrm{sing}M_k\cap B_{\rho_0}(0))=0,\quad k=1,2,\cdots,
			\end{split}
		\end{equation}
		suppose each $M_k$ is stable in $B_{\rho_0}(0)$ and suppose $\liminf\limits_{k\to  \infty}\CH^n(M_k\cap B_{\rho_0}(0))<\infty$. Then there exist a subsequence $\{k'\}\subset\{k\}$ and a varifold $V$ such that $M_{k'}\rightharpoonup V$ in $B_{\frac{\rho_0}{2}}(0)$ as varifold and such that
		\begin{equation}
			\begin{split}
				\mathrm{supp}V\cap B_{\frac{\rho_0}{2}}(0)=\bar{M}\cap B_{\frac{\rho_0}{2}}(0)
			\end{split}
		\end{equation}
	where $M$	is an  orientable minimal hypersurface with $\mathcal{H}^n(M\cap B_{\frac{\rho_0}{2}}(0))<\infty$ and $\mathcal{H}^{\alpha}( \mathrm{sing} M\cap B_{\frac{\rho_0}{2}}(0))=0$ for each  $\alpha>n-7$.
\end{theorem}	

	In remark 1 of \cite{SS81}, they studied $C^2$ minimal hypersurface in a $C^3$ Riemannian manifold. So, they arguments hold for stable hypersurfaces with bounded mean curvature(see equations (1.16) and (1.17) in \cite{SS81}). Consider a hypersurface $M$ such that 
	\begin{itemize}
		\item The mean curvature is bounded 
		\begin{equation}\label{eq3.3}
			|H|\leq \Lambda ,
		\end{equation}
		\item $M$ is stable, i.e.,
		\begin{equation}\label{eq3.4}
			\int_{M}|A|^2\xi^2 \leq \int_M |\nabla \xi|^2
		\end{equation}
		for any $C^1$ function $\xi$ with compact support in $M\cap B^{n+1}_{\rho_0}(0)$.
		\item An upper bound on the volume ratio
		\begin{equation}\label{eq3.5}
		\sup\limits_{x\in M_k}\sup\limits_{r>0}\frac{\mathcal{H}^n(M_k\cap B_r(x))}{\omega_n r^n}\leq  D
		\end{equation}
	for any $B_r(x)\subset B^{n+1}_{\rho_0}(0)$.
	\end{itemize}
Theorem 1 in \cite{SS81} can be written as 
	\begin{lemma}\label{lm3.2}
	Suppose $M$ is a hypersurface in $B_1(0)$ satisfying \eqref{eq3.3}-\eqref{eq3.5}. We also assume that $M$ is a $C^2$ embedded hypersurface with $\CH^{n-2}(\mathrm{sing}M)=0$. 
	There exists $\delta_0\in (0,\frac12)$ and $\beta\in (0,\frac14)$ depending only on $n,\Lambda, D$ such that if $X\in \bar{M}\cap B_{\frac14\rho}(0),\rho \in (0,\frac14)$, $M'$ is the connected  component of $\bar{M}\cap C(X,\rho)$ containing $X$, and 
	\begin{equation}
		\sup_{Y\in M'}|y_{n+1}-x_{n+1}|\leq \delta_0\rho,\quad \Lambda\rho\leq \delta_0,
	\end{equation}
	$($with $X=(x,x_{n+1}),Y=(y,y_{n+1}))$, then $M'\cap C(X,\frac12\rho)$ consists of a disjoint union of graphs of functions $u_1<u_2<\cdots<u_k$ defined on $B_{\frac12\rho}(x)\subset \IR^n$, satisfying the following estimate:
	\begin{equation}\label{eq3.7}
		\begin{split}
			\sup\limits_{B(x,\frac12\rho)}|Du_i| +\rho^{\beta}\sup\limits_{x,y\in B_{\frac12\rho}(0),x\neq y}\frac{|Du_i(x)-Du_i(y)|}{|x-y|^{\beta}}\leq C\delta_0.
		\end{split}
	\end{equation}
	for $i=1,2,\cdots,k,\beta \in (0,\beta_0)$ with $C$ depending only  on $n,\Lambda,D,\beta$.
	\end{lemma}
\begin{proof}
	The proof of Lemma~\ref{lm3.2} is subtle and of independent interest, and is therefore omitted here. For the reader’s convenience, a complete and rigorous proof is provided in \cite{Han-note}.
\end{proof}

\begin{proposition}\label{thm3.1}
	If $\{M^n_k\}\subset \mathbb{R}^{n+1}$ is a sequence of complete, connected and embedded hypersurfaces with  $\CH^{n-2}(M_k)=0$ and satisfying
	\begin{itemize}
		\item [(a)] $\sup\limits_{k\to \infty}\|\textbf{H}\|_{L^\infty(M_k)}\leq \Lambda_1$,
		\item [(b)] $\sup\limits_{x\in M_k}\sup\limits_{r>0}\frac{\mathcal{H}^n(M_k\cap B_r(x))}{\omega_n r^n}\leq  \Lambda_2$,
		\item [(c)] $\mathrm{index}(M_k)\leq I$.
	\end{itemize}
	for some fixed constant $\Lambda_1,\Lambda_2>0, I\in \mathbb{N}$. Then up to a subsequence, there exists a varifold $V$ such that $M_k$ converges to $V$ as varifold and such that
	\begin{equation}
		\begin{split}
			\mathrm{supp}V=\bar{M}
		\end{split}
	\end{equation} 
	where $M$ is a connected and oriented $C^{1,\beta}$ hypersurface in $\mathbb{R}^{n+1}$ and
	\begin{itemize}
		\item [$\mathrm(1)$] $\|\textbf{H}\|_{L^\infty(M)}\leq \Lambda_1$,
		\item [$\mathrm(2)$] $\sup\limits_{x\in M}\sup\limits_{r>0}\frac{\mathcal{H}^n(M\cap B_r(x))}{\omega_n r^n}\leq  \Lambda_2$,
		\item [$\mathrm(3)$] $\CH^{n-2}(\mathrm{sing}M)=0$.
	\end{itemize}
	We have that the convergence is in the $C^{1,\beta}$ topology for all $x\in \mathrm{reg}M\setminus\CY$ and $\beta\in (0,\beta_0)$, where $\CY\subset \bar{M}$ is a discrete set  such that $|\CY|\leq I$.
	
	Moreover, if  the convergence is $C^2$ in $\mathrm{reg}M\setminus\CY$, we have
	\begin{equation}
		\begin{split}
			\mathrm{index}(M)\leq I \text{ and }	\CH^{\alpha}(\mathrm{sing}M)=0	 \text{ for any }\alpha>n-7.
		\end{split}
	\end{equation}
\end{proposition}
\begin{proof}
The proof is deferred to Appendix \ref{secA}.
\end{proof}

	\subsection{Proof of Theorem \ref{prop4.2}.}

\begin{proof}[Proof of Theorem \ref{prop4.2}]
 By Proposition \ref{thm3.1} and \eqref{eq4.1}, there exists a stationary varifold  $\Sigma_\infty$ such that $\Sigma_{i,0}\rightharpoonup \Sigma_\infty$  as a varifold and the convergence is $C^{1,\alpha}$ at regular points of $\Sigma_\infty$ except a finite set $\CS$. For any $p\in \mathrm{reg}\Sigma_\infty\setminus \CS$, there exists $r>0$ such that 
	\begin{equation}
		\begin{split}
			r^{-n}\text{Vol}(B_r(p)\cap \Sigma_\infty)\leq (1+\theta/2),
		\end{split}
	\end{equation}
	where $\theta$ is  the constant in Theorem \ref{pseudo1}. So, for $i$ large enough, each connected component of $\Sigma_{i,0}\cap B_{r}(p)$ satisfies the condition in Theorem \ref{pseudo1}.	Then for any 
	$x\in \Sigma_{i,t}\cap B_{\varepsilon r}(p),t\in [-1,1]$, we have the estimate
	\begin{equation}
		|A|(x,t)\leq \frac{1}{\varepsilon r},
	\end{equation}
	where $\varepsilon$ is the constant in  Theorem \ref{pseudo1}.
	By Theorem \ref{mcf-cov}, $\Sigma_{i,t}$ smoothly converges to $\Sigma_\infty$(possible with multiplicity) in $B_{\varepsilon r}(p)$. So, again by  Proposition \ref{thm3.1}, we have 
	\begin{equation}
		\begin{split}
		\CH^{\alpha}(\mathrm{sing}\Sigma_\infty)=0
		\end{split}
	\end{equation}
	for $\alpha>n-7$. Moreover, $\mathrm{index}(\Sigma)\leq I$.
\end{proof}
\begin{remark}
	 	By the definition of convergence, the multiplicity is locally constant near any regular point. Since $\mathcal{H}^{n-2}(\operatorname{sing}\Sigma)=0$, the regular part of the limit hypersurface is connected; hence the multiplicity is constant on the regular part (cf.~\cite[Lemma 3.14]{LW19}).
\end{remark}

\section{Multiplicity-one convergence of the rescaled mean curvature flow}\label{sec4}
In this section, we show that a rescaled mean curvature flow with mean curvature exponential decay will converge smoothly to a stable cone with multiplicity
one.
\begin{theorem}\label{mul-one}
	Let $\{(\Sigma^n,\mathbf{x}(t)),0\leq t<+\infty\}$ be a closed rescaled mean curvature flow
	\begin{equation}\label{eq5.1}
		\bigg(\frac{\partial \mathbf{x}}{\partial t}\bigg)^\perp=-\bigg(H-\frac{1}{2}\langle \mathbf{x},\mathbf{n}\rangle \bigg)\mathbf{n}
	\end{equation}
	satisfying 
	\begin{equation}
		d(\Sigma_t,0)\leq D,\quad \max\limits_{\Sigma_t}|H(p,t)|\leq \Lambda_0e^{-\frac{t}{2}}\quad and \quad \max\limits_{t}\mathrm{ index}(\Sigma_t)\leq I_0
	\end{equation}
	for some constants $D,\Lambda_0,I_0$. Then for any sequence $t_i\to +\infty$ there exists a subsequence of $\{\Sigma_{t_i+t},-1<t<1\}$ such that it converges in the smooth topology to a stable minimal cone $\Sigma_\infty$ away from $\mathrm{sing}\Sigma_\infty$ with multiplicity one as $i\to +\infty$.
\end{theorem}

\subsection{ Convergence away from singularities}
\begin{lemma}\label{lm5.2}
Under the assumption of Theorem \ref{mul-one}, for any sequence $t_i\to +\infty$, there is a stable stationary cone $\Sigma_\infty\in \CC(N_0,n)$ and a finite set $\mathcal{S}_0\subset \mathrm{reg}(\Sigma_\infty)$ of points satisfying the following properties. For any $T>1/2$, there is a subsequence, still denoted by $\{t_i\}$, such that $\{\Sigma_{t_i+t},-T<t<T\}$ converges
in the smooth topology, possibly with multiplicities at most $N_0$, to $\Sigma_\infty$ away from $\{(x,t)|t\in (-T,T),x\in e^{\frac{t}{2}}\mathcal{S}_0\}\cup \{(x,t)|t\in (-T,T),x\in \mathrm{sing}\Sigma_\infty\}$.
\end{lemma}
\begin{proof}
 For any $t_i\to +\infty$, we can obtain a refined sequence converging to a limit stationary varifold $\tilde{\Sigma}_\infty$. In fact, for any sequence $t_i\to +\infty$, we
can rescale the flow $\Sigma_t$ by
\begin{equation}
	s=1-e^{-(t-t_i)},\quad \tilde{\Sigma}_{i,s}=\sqrt{1-s}\Sigma_{t_i-\log(1-s)}
\end{equation}
such that for each $i$ the flow $\{\tilde{\Sigma}_{i,s},1-e^{t_i}\leq s<1\}$ is a mean curvature flow such that :
\begin{itemize}
	\item [(a)] For any small $\lambda>0$, the mean curvature of $\tilde{\Sigma}_{i,s}$ satisfies
	\begin{equation}
	\lim\limits_{i\to +\infty}\max\limits_{\tilde{\Sigma}_{i,s}\times[1-e^{t_i},1-\lambda]}|\tilde{H}_i(p,s)|=0;
	\end{equation}
	\item [(b)]   Uniform upper bound on  the index of $\tilde{\Sigma}_{i,s}$;
	\item [(c)] An upper bound holds for the volume ratio;
	\item [(d)]  A lower bound for the volume ratio;
	\item [(e)] There exists a constant $D^\prime>0$ such that $ d(\tilde{\Sigma}_{i,s}, 0) \leq D^\prime$ for any $i$.
\end{itemize}
For any $T_0>2$, small $\lambda\in (0,1)$ and any $s_0\in [-T_0+1,-\lambda]$ the sequence $\{\tilde{\Sigma}_{i,s_0+\tau},-1<\tau<1\}$ is a refined sequence. By Proposition \ref{prop4.2}, a subsequence of $\{\tilde{\Sigma}_{i,s_0+\tau},-1<\tau<1\}$ converges in smooth topology, possibly with multiplicity at most $N_0$, to a stationary varifold $\tilde{\Sigma}_\infty$ away from $\mathrm{sing}\Sigma_\infty$ and a finite set of points $\tilde{\mathcal{S}}=\{q_1,\dots,q_l\}$ . Moreover, $\tilde{\Sigma}_\infty$ and $\tilde{\mathcal{S}}$ are independent
of $T_0$ and $\lambda$.

 Each limit $\tilde{\Sigma}_\infty$ must be a stable cone at origin. In fact, by
Huisken’s monotonicity formula, along the rescaled mean curvature flow \eqref{eq5.1} we have 
\begin{equation}
\int_0^\infty\int_{\Sigma_t}e^{-\frac{|x|^2}{4}}\bigg|H-\frac{1}{2}\langle \mathbf{x}, \mathbf{n}\rangle\bigg|^2d\mu_tdt< \int_{\Sigma_0}e^{-\frac{|x|^2}{4}}<+\infty.
\end{equation}
Noting that  for each  $i$ the flow $\{\tilde{\Sigma}_{i,s},1-e^{t_i}\leq s<1\}$ is a mean curvature flow and we denote the solution by  $\tilde{\mathbf{x}}_{i,s}$. Therefore, for fixed $T_0 > 0$, small $\lambda>0$ and large
$i$ we have
\begin{equation}
\begin{split}
&\lim\limits_{i\to +\infty}\int_{-T_0}^{1-\lambda}ds\int_{\Sigma_t}e^{-\frac{|\tilde{\mathbf{x}}_{i,s}|^2}{4(1-s)}}\bigg|\tilde{H}_{i,s}-\frac{1}{2(1-s)}\langle \tilde{\mathbf{x}}_{i,s}, \mathbf{n}\rangle\bigg|^2d\tilde{\mu}_{i,s}\\
&=\lim\limits_{t_i\to +\infty}\int_{t_i-\log(1+T_0)}^{t_i-\log\lambda}\int_{\Sigma_t}e^{-\frac{|\mathbf{x}|^2}{4}}\bigg|H-\frac{1}{2}\langle \mathbf{x}, \mathbf{n}\rangle\bigg|^2d\mu_t=0.
\end{split}
\end{equation}
Since $\{\tilde{\Sigma}_{i,s},-T_0<s<1-\lambda\}$ converges locally smoothly, possibly with
multiplicity at most $N_0$, to $\tilde{\Sigma}_\infty$ away from $\mathrm{sing}(\tilde{\Sigma}_\infty)\cup\CS_0$, we have
\begin{equation}
	\begin{split}
		\lim\limits_{i\to +\infty}\int_{-T_0}^{1-\lambda}ds\int_{\Sigma_t}e^{-\frac{|\tilde{\mathbf{x}}_\infty|^2}{4(1-s)}}\bigg|\tilde{H}-\frac{1}{2(1-s)}\langle \tilde{\mathbf{x}}_\infty, \mathbf{n}\rangle\bigg|^2d\tilde{\mu}_{\infty,s}=0.
	\end{split}
\end{equation}
Therefore, $\{(\tilde{\Sigma}_\infty,\tilde{\mathbf{x}}_\infty(p,s)),-T_0<s<1-\lambda\}$ satisfies the equation
\begin{equation}
	\tilde{H}-\frac{1}{2(1-s)}\langle \tilde{\mathbf{x}}_\infty,\mathbf{n} \rangle=0
\end{equation}
away from singular set $\mathrm{sing}(\tilde{\Sigma}_\infty)\cup \CS_0$. Since $\tilde{\Sigma}_\infty$ is a  stationary varifold with  index $\leq I_0$, we know that $\tilde{H}=\langle \tilde{\mathbf{x}}_\infty,\mathbf{n}\rangle=0$ and $\tilde{\Sigma}_\infty$ is a stable stationary cone(Lemma \ref{stable-cone}). Let $\mathbf{x}_i(p,t)=\mathbf{x}_i(p,t_i+t)$ and $\Sigma_{i,t}=\Sigma_{t_i+t}$. Since $\{\tilde{\Sigma}_{i,s}-T_0<s<1-\lambda\}$
 converges locally smoothly to $\tilde{\Sigma}_\infty$ away from $\mathrm{sing}\tilde{\Sigma}_\infty\cup\CS_0$, the flow $\{\Sigma_{i,t},-\log(1+T_0)<t<-\log\lambda\}$ also  converges locally smoothly to $\tilde{\Sigma}_\infty$ away from  $\mathrm{sing}\tilde{\Sigma}_\infty\cup\CS$ where ${\mathcal{S}}=\{(x,t)|t\in (-\log(1+T_0)<t<-\log\lambda),x\in e^{\frac{t}{2}}\CS_0\}$. Moreover, $\tilde{\Sigma}_\infty$ and $\CS_0$  are independent of $T_0$ and $\lambda$. The lemma is proved.
\end{proof}
\subsection{Decomposition of spaces }$\\$

In this subsection, we follow the argument in \cite{LW19} to decompose the space and define an almost “monotone decreasing” quantity, which will be used to select time slices such that the limit self-shrinker is L-stable. First, we decompose the space as follows.

\begin{definition}\label{def5.4}
Given a hypersurface $M$ and $x\in \IR^{n+1}$, we define the regularity scale $r_M$ as the supremum of $0\leq r\leq 1$ such that 
	\begin{equation}
		\begin{split}
			\sup\limits_{y\in M\cap  B_r(y)}r|A|(y)\leq 1,
		\end{split}
	\end{equation}
	where $A$ is the second fundamental form. If $y\in \mathrm{sing}M$, we  assume that $|A|(y)=+\infty$.
\end{definition}
\begin{definition}\label{dec}
Fix large $R>0$ and small $\varepsilon>0$.
\begin{itemize}
	\item [(1)] We define the set $\mathbf{S}=\mathbf{S}(\Sigma_t,\varepsilon,R)=\{y\in \Sigma_t||y|<R,|A|_{\Sigma_t}(y)>\varepsilon^{-1}\}$.
	\item [(2)] The ball $B_R(0)$ can be decomposed into three parts as follows:
\end{itemize}
\end{definition}
\begin{itemize}
	\item  the high curvature part $\mathbf{H}$, which is defined by
	\begin{equation*}
		\begin{split}
			\mathbf{H}=\mathbf{H}(\Sigma_t,\varepsilon,R)=\biggl\{x\in \IR^{n+1}\bigg||x|<R,d(x,\mathbf{S})<\frac{\varepsilon}{2}\biggr\}
		\end{split}
	\end{equation*}
	\item the thick part $\mathbf{TK}$, which is defined by 
	
	\begin{equation*}
		\begin{split}
		\quad	\mathbf{TK}&=\mathbf{TK}(\Sigma_t,\varepsilon,\zeta,R)\\
			&=\biggl\{x\in \IR^{n+1}\bigg| |x|<R, \text{ there is a continuous curve }\gamma\subset B_R(0)\setminus(\mathbf{H}\cup \Sigma_t)\\
			&\qquad \text{ connecting  } x \text{ and some } y \text{ with } B(y,\zeta)\subset B_R(0)\setminus(\mathbf{H}\cup \Sigma_t)\biggr\},
		\end{split}
	\end{equation*}
	\item  the thin part $\mathbf{TN}$, which is defined by $\mathbf{TN}=\mathbf{TN}(\Sigma_t,\varepsilon,\zeta,R)=B_R(0)\setminus(\mathbf{H}\cup \mathbf{TK}) $.
\end{itemize}
 The high-curvature zone $\mathbf{H}$ is a neighborhood of points exhibiting a large second fundamental form. 
The thin region $\mathbf{TN}$ occupies the space between the top and bottom sheets. 
The thick region $\mathbf{TK}$, on the other hand, is the union of path-connected components of the region lying outside the sheets(cf. figure \ref{fig3}).

\begin{figure}[htbp]
    \centering
    \includegraphics[width=0.5\textwidth]{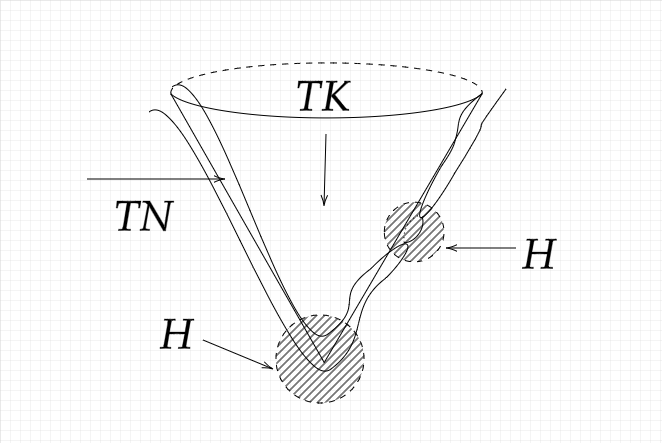}
    \caption{Decomposition of spaces.}
    \label{fig3}
\end{figure}

In the rest of this subsection, we concentrate on the quantity $|\mathbf{TN}|$. 
We first establish that $|\mathbf{TN}| \to 0$ as $t \to \infty$; 
subsequently, we choose a sequence $t_i \to \infty$ such that the value $|\mathbf{TN}|(t_i)$ serves as a uniform upper bound for all subsequent values.

Since every sequence in $\{\Sigma_t\}$ admits a subsequence converging to a stable cone, we first prove that, for an appropriate parameter $\zeta$, the volume of the corresponding $\mathbf{TN}$ vanishes; equivalently, $\mathbf{TN}$ is empty.
\begin{definition}
	For any $N>0$, we denote by $\CC(N,n)$ the space of all stable minimal cones $\mathbf{C}\subset\IR^{n+1}$ satisfying $\CH^{n-2}(\mathbf{C})=0$, $\Theta(\mathbf{C},x)=1$ for a.e. $x\in \mathbf{C}$ and
	\begin{equation}
		\begin{split}
			\Theta(\mathbf{C},0)=\frac{\mathrm{Vol}(B_r(0)\cap \mathbf{C})}{\omega_n r^n}\leq N.
		\end{split}
	\end{equation}
\end{definition}
\begin{proposition}
$\CC(N,n)$ 	is compact in the varifold sense.
\end{proposition}
\begin{proof}
	Suppose that $\{\mathbf{C}_i\}\in \CC(N,n)$ is a  sequence of stable stationary cones. Then, there exists  a $\mathbf{C}\in \CC(N,n)$ such that $\mathbf{C}_i$ converges to $\mathbf{C}$ with multiplicity $m$. Assume that $\mathrm{link}(\mathbf{C}_i):=\mathbf{C}_i\cap \partial B^{n+1}_1(0)=\Gamma_i$ and $\mathrm{link}(\mathbf{C})=\Gamma$.  Hence, $\Gamma_i$ converges to $\Gamma$ with multiplicity $m$. Noting that $\mathbf{C}_i$ and $\mathbf{C}$ are stationary in $\IR^{n+1}$, $\Gamma_i$ and $\Gamma$ are stationary in $\IS^{n}(1)$.  Since $\CH^{n-7+\alpha}(\mathrm{sing}\mathbf{C}_i)=\CH^{n-7+\alpha}(\mathrm{sing}\mathbf{C})=0$ for any $\alpha>0$ and $x\in \mathrm{reg}(\Gamma_i)$ if  $x\in \mathrm{reg}(\mathbf{C}_i)$ , $\CH^{n-7+\alpha}(\mathrm{sing}\Gamma_i)=\CH^{n-7+\alpha}(\mathrm{sing}\Gamma)=0$ for any $\alpha>0$. If $x$ is a regular point of $\Gamma$, then $x$ is a regular point of $\mathbf{C}$. So, by Theorem \ref{thm3.2}, $\Gamma_i$ smoothly converges to $\Gamma$ at $x$ in $\IS^{n}$. Lemma \ref{lm5.7} implies $m=1$, which completes the proof.
\end{proof}
We next introduce a function $s$ to rule out the occurrence of a multi-sheeted 
phenomenon on the limiting cone. Consequently, any point outside $\mathbf{H}$ 
can be connected to a point away from $\Sigma$, and hence we conclude that $\mathbf{TN}=\emptyset$(cf. figure \ref{fig5}).
\begin{lemma}\label{lm5.8}
	Let  $R,N,\varepsilon>0$. For any $\mathbf{C}\in \CC(N,n)$ and $x\in \mathrm{reg}  \mathbf{C}$, we define the supremum of the radius $s$ such that 
	\begin{equation}
		\begin{split}
			B_s(x+s\mathbf{n}(x))\cap \mathbf{C}=\emptyset,\quad B_s(x-s\mathbf{n}(x))\cap \mathbf{C}=\emptyset,
		\end{split}
	\end{equation}
	where $\mathbf{n}(x)$ denotes the normal vector of $\mathbf{C}$ at $x$. Then there exists $\zeta_0(R,N,\rho)>0$ such that for any $\mathbf{C}\in\CC(N,\rho)$ and $x\in (\mathbf{C} \cap B_R(0))\setminus \mathbf{H}(\mathbf{C},\varepsilon,R) $ we have
	\begin{equation}
		\begin{split}
			s(x;\CC)\geq \zeta_0.
		\end{split}
	\end{equation}
\end{lemma}
\begin{proof}
	Suppose not. For any $j$ there exist $\mathbf{C}_j$ and $x_j\in (\mathbf{C}_j \cap B_R(0))\setminus \mathbf{S}(\mathbf{C}_j,\varepsilon,R)$, such that
	\begin{equation}\label{sep}
		\begin{split}
			s(x_j;\mathbf{C}_j)\leq j^{-1}.
		\end{split}
	\end{equation} 
	 Assume that $\mathbf{C}_j\rightharpoonup \mathbf{C}_\infty$ in varifold sense  and $x_j\to x_\infty$ as $j\to \infty$. Since $x_j\in (\mathbf{C}_j \cap B_R(0))\setminus \mathbf{H}(\mathbf{C}_j,\varepsilon,R)$, we know that $\mathbf{C}_j$ smoothly converges to $\mathbf{C}$ near $x_\infty$. \eqref{sep} implies that $s_{\mathbf{C}_\infty}(x_\infty)\leq\liminf_{j\to \infty} s(x_j;\mathbf{C}_j)=0 $, which contradicts that $s(x_\infty;\mathbf{C}_\infty)>0$.
\end{proof}

\begin{figure}[htbp]
    \centering
    \includegraphics[width=0.4\textwidth]{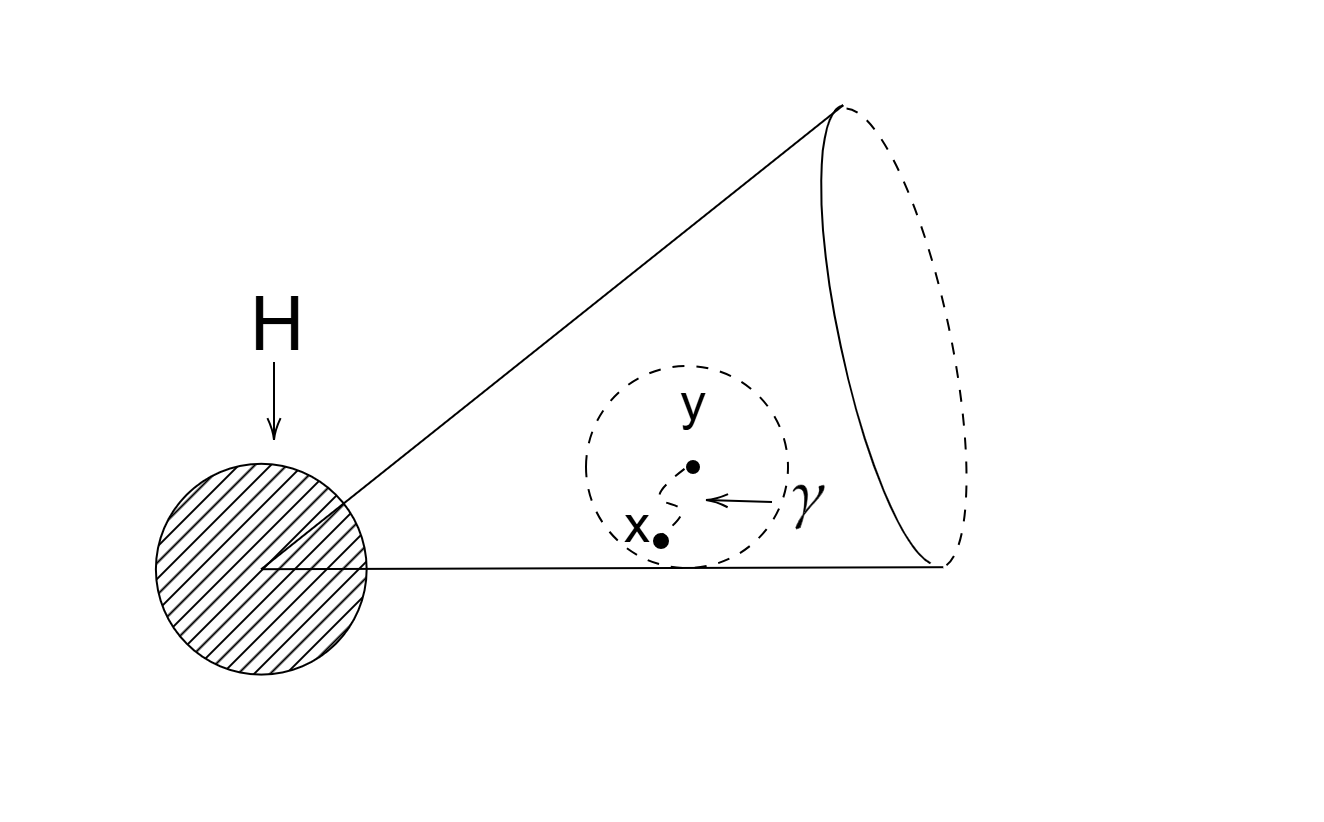}
    \caption{Separation}
    \label{fig5}
\end{figure}
A direct corollary of Lemma \ref{lm5.8} is the following result

\begin{lemma}\label{lm5.9}
For any $R,N,\varepsilon>0$, there exists a constant $\zeta_0(R,N,\rho,\varepsilon)>0$ such that for any $\zeta\in (0,\zeta_0)$, we have
\begin{equation}\label{eq5.25}
	\begin{split}
		|\mathbf{TN}(\mathbf{C},\varepsilon,\zeta,R)|=0\quad \text{for all }\mathbf{C}\in \CC(N,n).
	\end{split}
\end{equation}
Here the notation $|\Omega|$ denotes the volume of $\Omega$ with respect to the standard metric on $\IR^{n+1}$.
\end{lemma}
\begin{proof}
For any  $R,N,\varepsilon>0$, we choose $\zeta$ the same constant in Lemma  \ref{lm5.8}. Thus, equation \eqref{eq5.25} follows from Lemma  \ref{lm5.8} and the definition of $\mathbf{TN}$.
\end{proof}
Using Lemma \ref{lm5.9} we show that the quantity $\mathbf{TN}$ along the flow will tend to zero.
\begin{lemma}\label{lm5.10}
	Fix $R,N,\varepsilon>0$. Under the assumption of Theorem \ref{mul-one}, there exists a constant $\zeta_0(R,N,n,\varepsilon)>0$ such that for any  $\zeta\in (0,\zeta_0)$, we have 
	\begin{equation}
		\begin{split}
			\lim\limits_{t\to \infty}|\mathrm{TN}(\Sigma_t,\varepsilon,\zeta,R)|=0.
		\end{split}
	\end{equation}
\end{lemma}
\begin{proof}
	For any $\zeta>0$, by definition \ref{dec} we have
	\begin{equation}
		\begin{split}
			\mathbf{TN}(\Sigma_{t_i},\varepsilon,\zeta,R)\to \mathbf{TN}(\Sigma_\infty,\varepsilon,\zeta,R)\setminus B_{\frac{\varepsilon}{2}}(\CS_0),
		\end{split}
	\end{equation}
in the local Hausdorff sense(cf. figure \ref{fig3}).	Here $B_{\varepsilon}(\CS_0)=\cup_{p\in \CS_0}B_\varepsilon(p)$. Therefore, by Lemma \ref{lm5.9} we have
	\begin{equation}
		\begin{split}
			\lim\limits_{t_i\to +\infty}|\mathbf{TN}(\Sigma_{t_i},\varepsilon,\zeta,R)|\leq |\mathbf{TN}(\Sigma_{\infty},\varepsilon,\zeta,R)|=0,
		\end{split}
	\end{equation}
	where $\zeta\in (0,\zeta_0)$ and $\zeta_0$ is the constant in Lemma \ref{lm5.9}. The lemma is proved.
\end{proof}

Since the volume tends to zero, we can choose a sequence time $\{t_i\}$ such that (cf. \cite[Lemma 4.7]{LW19}):
\begin{lemma}\label{lm5-9}
	Fix $R>0$ and $\tau\in (0,1)$. Let $\{t_i\}$ be any sequence as in Lemma \ref{lm5.2}. If the
	multiplicity of the convergence in Lemma \ref{lm5.2} is more than one, then for any $\varepsilon>0,\zeta\in (0,\zeta_0)$, there exists $i_0>0$ such that for any $i\geq i_0$ we have 
	\begin{equation}
		\begin{split}
			\inf\limits_{t\in [t_i-\tau,t_i]}|\mathbf{TN}(\Sigma_{t},\varepsilon,\zeta,R)|>0.
		\end{split}
	\end{equation}
\end{lemma}
\begin{proof}
Since $M_t$ is embedded and $\{\Sigma_{t_i+t},-\tau\leq t\leq\tau\}$ converges locally smoothly to the limit stable cone $\Sigma_\infty$, all components of $(\Sigma_t\cap B_R(0))\setminus \mathbf{H}(\Sigma_t,\varepsilon,\zeta,R)$ with $t\in [t_i-\tau,t_i]$ lie in the $B_{\zeta/2}$ neighborhood of $\Sigma_\infty$. By the definition of $\mathbf{TN}$, for any $t\in [t_i-\tau,t_i]$ the
quantity $\mathbf{TN}(\Sigma_{t},\varepsilon,\zeta,R)$ is nonempty and we have $|\mathbf{TN}(\Sigma_{t},\varepsilon,\zeta,R)|>0$.
\end{proof}

Since the volume of the thin part tends to zero, we can choose a sequence time $\{t_i\}$ such that (cf. \cite[Lemma 4.8]{LW19}).
\begin{lemma}\label{lm5.12}
	Let $R,\varepsilon,\tau>0$, $\zeta\in (0,\zeta_0)$ and $f(t,\zeta)=\inf\limits_{s\in [t-\tau,t]}|\mathrm{TN}(\Sigma_s,\varepsilon,\zeta,R)|$. For any $t_0>0$ and $l>0$,we can find a sequence $\{t_i\}$ with $t_{i+1}>t_i+l$ such that for any $i\in \IN$,
	\begin{equation}
		\begin{split}
			\sup\limits_{t\in [t_i,t_i+l]}f(t,\zeta)\leq 2f(t_i,\zeta).
		\end{split}
	\end{equation}
\end{lemma}
\subsection{Construction of auxiliary functions}$\\$

We are ready to construct a positive immortal solution to $\partial_t u-Lu=0$. Then we use it to show the $L$-stability of
the limit cone. We fix $R,T>1$. For any sequence $t_i\to +\infty$, by Lemma \ref{lm5.2} a subsequence of $\{\Sigma_{i,t},-T<t<T\}$ converges in smooth topology to a cone $\Sigma_\infty$ away from $\mathrm{sing}\Sigma_\infty\cup\CS$. We assume that the multiplicity of the convergence is a constant $N_0\geq 2$. As in \cite{LW19}, we
construct some functions as follows:

Let  $\varepsilon>0$ and large $R>0$. We define
\begin{equation}
	\begin{split}
		\Omega_{\varepsilon,R}=\{x\in \mathrm{spt}\Sigma_\infty:|r_{\Sigma_\infty}|(x)\geq \varepsilon^{-1}\}\cap (B_R(0)\setminus B_\varepsilon(0)).
	\end{split}
\end{equation}
\begin{lemma}
There exists $t^*>0$ depending  only on $\varepsilon$ and $\CS_0$ such that 
\begin{equation}
	\begin{split}
			\Omega_{\varepsilon,R}\cap \CS_t=\emptyset,\quad \forall\; t>t^*.
	\end{split}
\end{equation}
\end{lemma}
\begin{proof}
Since $\CS_t=e^{t/2}\CS_0$,	all points in $\CS_t\setminus \{0\}$ will move further away from the point $\{0\}$ when $t$ is increasing. Set
	\begin{equation}
		\begin{split}
			\min\limits\CS_0:=\mathrm{min}\{|y||y\in \CS_0\setminus \{0\}\}>0, \quad t^*:=2\log \frac{R}{ \mathrm{min}\CS_0}.
		\end{split}
	\end{equation}
	Then for any $t\geq t^*$ and $y_0\in \CS_0\setminus \{0\}$ we have $e^{\frac{t}{2}}|y|\geq R$. This implies  that
	\begin{equation}
		\begin{split}
			(\CS_t\setminus \{0\})\cap B_{R}(0)=\emptyset,\quad \forall \;t>t^*.
		\end{split}
	\end{equation}
\end{proof}
	Let $u^+_i(x,t) $ and $u^-_i(x,t)$ be the graph functions representing the top and bottom sheet over $\Omega_{\varepsilon,R}$ (see figure \ref{fig4}).
    \begin{figure}[htbp]
    \centering
    \includegraphics[width=0.6\textwidth]{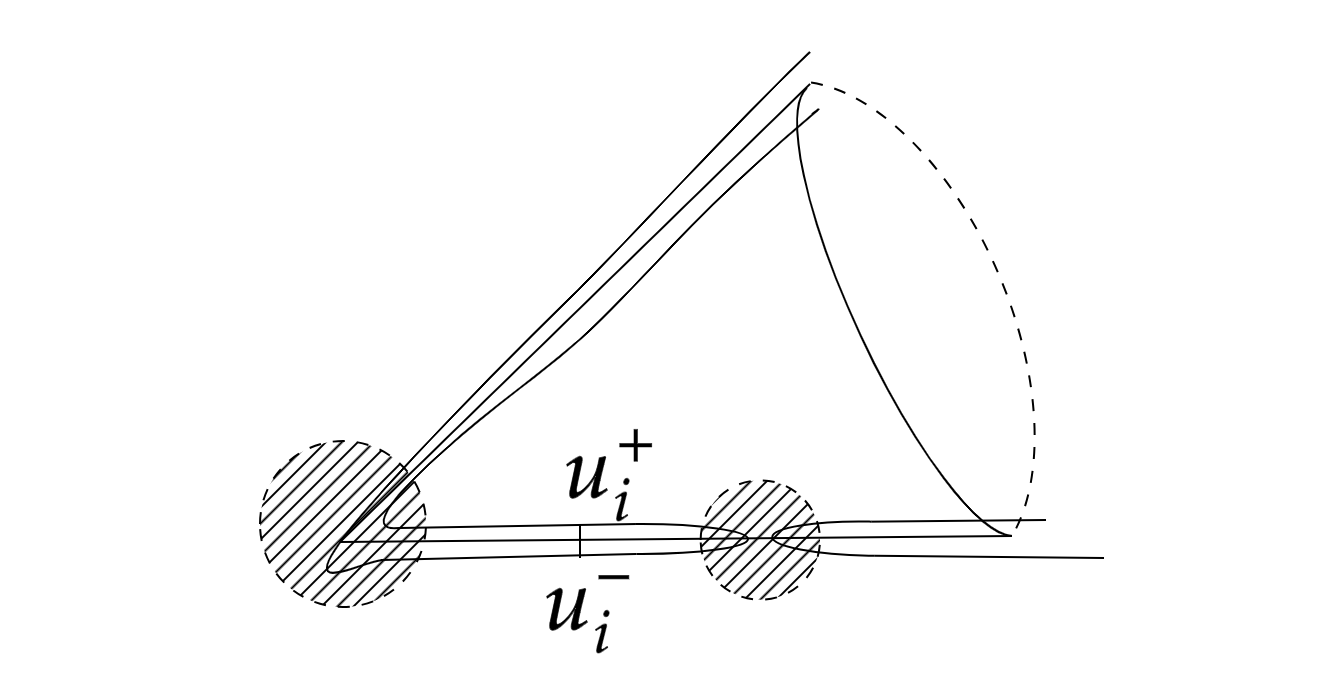}
    \caption{Height function.}
    \label{fig4}
\end{figure}
By the convergence property of the flow $\{(\Sigma_{i,t},\mathbf{x}_i),-T<t<T\}$, for any $\varepsilon>0$ and large $R$ there exists $i_0>0$ such that for any $i\geq i_0$ and $t\in (t_R,T)$ the function $u^+_i(x,t)$ and $u^-_i(x,t)$ are well-defined on $\Omega_{\varepsilon,R}$.  Similar to the Appendix of \cite{LW22}, we know that the height difference function $u_i(x,t)=u^+_i(x,t)-u^-_i(x,t)$ satisfies the following parabolic equations on $\Omega_{\varepsilon,R}\times I$
		\begin{equation}
	\frac{\partial u_i}{\partial t}=\Delta u_i-\frac{1}{2}\langle x,\nabla u_i\rangle+|A|^2u_i+\frac{1}{2}u_i+a_i^{pq}u_{i,pq}+b_i^pu_{i,p}+c_iu_i
	\end{equation}
	where $\Delta$ denotes the Laplacian operator on $\mathrm{reg}(\Sigma_\infty)$. The	coefficients $a_i^{pq},b_i^p,c_i$ are small and tend to zero as $u_i^+$ and $u_i^-$ tend to zero as $t_i\to +\infty$.

	Clearly, $u_i$ are positive solutions satisfying $u_i\to 0$. We fix a point $x_0\neq \CS_t\cup \{0\}$ and define the normalized height difference function
	\begin{equation}
		w_i(x,t):=\frac{u_i(x,t)}{u_i(x_0,1)}.
	\end{equation}
	Then $w_i(x, t)$ is a positive function on $\Omega_{\varepsilon,R}$ with $w_i(x_0,1)=1$ and
	\begin{equation}
	\frac{\partial w_i}{\partial t}=\Delta w_i-\frac{1}{2}\langle x,\nabla w_i\rangle+|A|^2w_i+\frac{1}{2}w_i+a_i^{pq}w_{i,pq}+b_i^pw_{i,p}+c_iw_i.
	\end{equation}
    
Away from the high curvature neighborhood $\mathbf{H}$, the flow smoothly converges 
to the limit cone with multiplicity. One can show that, for large $t$, the integral 
of the height function is comparable to the volume $|\mathbf{TN}|$ (cf.~\ref{fig4}). 
Since $\Sigma_{i,t}=\Sigma_{t_i+t}$, the same estimate holds for $\Sigma_{i,t}$ 
when $t_i$ is sufficiently large.

	\begin{lemma}\cite[Lemma 3.14]{LW22}\label{lm5.11}
	Fix $\varepsilon,R$ and $T$ as above. For any sequence $\{t_i\}$ chosen in Lemma \ref{lm5.2}, there exists $t_T>0$ such that  for any $t\in (-T,T)$ and $t_i>t_T$ we have
	\begin{equation}\label{eq5.22}
		\frac{1}{2}\int_{\Omega_{\varepsilon,R}}u_i(x,t)d\mu_\infty\leq |\mathbf{TN}(\Sigma_{i,t})|\leq 2\int_{\Omega_{\frac{1}{5}\varepsilon,R}}u_i(x,t)d\mu_\infty,
	\end{equation}
	where $d\mu_\infty$ denotes the volume form of $\Sigma_\infty$.
	\end{lemma}
Combining Lemma \ref{lm5.12} and Lemma \ref{lm5.11}, we obtain an almost monotonicity estimate for
$\|w_i\|_{L^1}$.
Since $w_i$ satisfies a parabolic equation, the parabolic Harnack inequality
provides a $C^0$ upper bound for $w_i$ in terms of its $L^1$ norm.
The latter is uniformly bounded by the initial data.

Li and Wang \cite{LW22} established a Harnack inequality for smooth shrinkers. In our case, although the shrinker may have singularities, we only apply the Harnack inequality on the regular part, where the curvature remains bounded. Consequently, the argument follows exactly the same lines as their proof, with no essential differences. Therefore, we obtain the following estimates.
\begin{lemma}~\cite[Lemma 3.17]{LW22}\label{lm-4.14}
	For any compact set  $K\subset\subset \Omega_{\varepsilon,R}$, $T>t^*+2$, there exist two constants 
	\begin{equation}
		\begin{split}
			C_1=C_1(\varepsilon,K,\CS_0,x_0)>0,\quad C_2=C_2(\varepsilon,\CS,T,x_0)>0
		\end{split}
	\end{equation}
	such that for large  $t_i$, we have
	\begin{equation}
		\begin{split}
			C_2(\varepsilon,\CS,T,x_0)<w_i(x,t)<C_1(\varepsilon,K,\CS_0,x_0),\forall (x,t)\in K\times [t^*+1,T-1].
		\end{split}
	\end{equation}
	Moreover, at time $t=t^*+1$ there exists $C_3=C_3(\varepsilon,\CS_0,x_0)>0$ independent of $T$ such that 
	\begin{equation}
		\begin{split}
			w_i(x,t_\varepsilon+1)\geq C_3(\varepsilon,\CS_0,x_0)>0\quad \forall x\in K.
		\end{split}
	\end{equation}
\end{lemma}
Passing to the limit as $i\to\infty$, we note that the coefficients $a_i^{pq}$, $b_i^p$, and $c_i$ converge to zero as $u_i^+$ and $u_i^-$ tend to zero. Therefore, the sequence $w_i$ converges to a positive immortal solution $w$ satisfying $\partial_t w-Lw=0.$ Moreover, the estimates in Lemma~\ref{lm-4.14} pass to the limit (cf.~\cite[Proposition 4.9]{LW19}).
	\begin{proposition}\label{prop5.13}
	Under the assumption of Theorem \ref{mul-one}, if the multiplicity of
	a convergent sequence is at least two, we can
	find a smooth function $w$ defined on $\Omega_{\varepsilon,R}\times [0,\infty)$ such that
	\begin{equation}\label{eq-5.13}
		\frac{\partial w}{\partial t}=Lw:=\Delta w-\frac{1}{2}\langle x,\nabla w\rangle+|A|^2w+\frac{1}{2}w,
	\end{equation}
	where $\Delta$ is the Laplacian operator on $\mathrm{reg}\Sigma$. Furthermore,
for any compact set  $K\subset\subset \Omega_{\varepsilon,R}$	there is a constant $C=C(K,\varepsilon,R)>0$ independent of $t$ such that
	\begin{equation}\label{eq5-3-27}
		\begin{split}
			0<&w(x,t)<C,\quad \forall (x,t)\in K\times [0,\infty),\\
			&w(x,0)\geq \frac{1}{C},\quad x\in K.
		\end{split}
	\end{equation}
   \end{proposition}

\subsection{Proof of the multiplicity-one convergence}$\\$

We next use the renormalized function $w$ to show that if the multiplicity is greater than one, then the limit shrinker is $L$-stable, which yields a contradiction.

%In this subsection, we shall show Theorem \ref{mul-one}, i.e., the limit cone is multiplicity one.
\begin{lemma}\label{lm5.16}
	For any function $\varphi\in W^{1,2}_c(\mathrm{reg}(\Sigma_\infty)\setminus\{0\})$ we have 
	\begin{equation}\label{eq-5.37}
		-\int_{\Sigma_\infty}(\varphi L\varphi)e^{-\frac{|x|^2}{4}}\geq 0.
	\end{equation}
	Here, $W^{1,2}_c(\mathrm{reg}\Sigma_\infty\setminus\{0\})$ denotes the set of all functions with compact
	support in $\mathrm{reg}\Sigma_\infty\setminus\{0\}$ and $L$ is the operator defined by \eqref{eq5.12} on $\mathrm{reg}\Sigma_\infty$.
\end{lemma}
\begin{proof}
Given a function $\varphi \in W^{1,2}_c(\text{reg}(\Sigma_\infty)\setminus\{0\})$, there exist
$\varepsilon,R>0$ such that
\begin{equation}
\operatorname{supp}\varphi:=K \subset\subset \Omega_{\varepsilon,R}.
\end{equation}

For such $\varepsilon$ we choose $\{t_i\}$ as in Lemma \ref{lm5.12} and we denote by $\Sigma_\infty$ the limit cone of the sequence $\{\Sigma_{t_i}\}$. In light of Proposition \ref{prop5.13},  we obtain a positive function $w$ defined on $\Omega_{\varepsilon,R}\times [0,\infty)$. Let $v:=\log w$. It follows from \eqref{eq-5.13} that $v$ satisfies the equation
\begin{equation}
	\frac{\partial v}{\partial t}=\Delta v-\frac{1}{2}\langle x,\nabla v\rangle+|A|^2+\frac{1}{2}+|\nabla v|^2\quad \forall (x,t)\in \Omega_{\varepsilon,R}\times [0,\infty).
\end{equation}

Then integration by parts implies that
\begin{equation}
	\begin{split}
	0&=\int_{\Sigma_{\infty}}\text{div}\left(\varphi^2 e^{-\frac{|x|^2}{4}}\nabla v\right)\\
	&=\int_{\Sigma_{\infty}}\left(2\varphi\langle \nabla\varphi,\nabla v\rangle+\varphi^2(\Delta v-\frac{1}{2}\langle x,\nabla v\rangle)\right)e^{-\frac{|x|^2}{4}}\\
	&=\int_{\Sigma_{\infty}}\left(2\varphi\langle \nabla\varphi,\nabla v\rangle+\varphi^2(\frac{\partial v}{\partial t}-|A|^2-\frac{1}{2}-|\nabla v|^2)\right)e^{-\frac{|x|^2}{4}}\\
	&\leq \int_{\Sigma_{\infty}}\left(|\nabla\varphi|^2+\varphi^2\frac{\partial v}{\partial t}-(|A|^2+\frac{1}{2})\varphi^2\right)e^{-\frac{|x|^2}{4}}.
	\end{split}
\end{equation}
The above inequality can be rewritten as
\begin{equation}\label{eq-5.41}
\begin{split}
-\int_{\Sigma_{\infty}}\left(\varphi L\varphi\right)e^{-\frac{|x|^2}{4}}&= \int_{\Sigma_{\infty}}\left(|\nabla\varphi|^2-(|A|^2+\frac{1}{2})\varphi^2\right)e^{-\frac{|x|^2}{4}}\\
&\geq \frac{d }{d t}\int_{\Sigma_{\infty}}v\varphi^2e^{-\frac{|x|^2}{4}}.
\end{split}
\end{equation}
Fix $T>0$. Integrating both sides of \eqref{eq-5.41}, by \eqref{eq5-3-27}, we have
\begin{equation}
\begin{split}
-\int_{0}^Tdt\int_{\Sigma_{\infty}}\left(\varphi L\varphi\right)e^{-\frac{|x|^2}{4}}&\geq \int_{\Sigma_{\infty}}v\varphi^2e^{-\frac{|x|^2}{4}}\bigg|_{t=0}-\int_{\Sigma_{\infty}}v\varphi^2e^{-\frac{|x|^2}{4}}\bigg|_{t=T}\\
&= \int_{K}v\varphi^2e^{-\frac{|x|^2}{4}}\bigg|_{t=0}-\int_{K}v\varphi^2e^{-\frac{|x|^2}{4}}\bigg|_{t=T}\\
&\geq -C.
\end{split}
\end{equation}
where $C$ is independent of $T$. Thus, we have 
\begin{equation}\label{eq-5.43}
	-\int_{\Sigma_{\infty}}\left(\varphi L\varphi\right)e^{-\frac{|x|^2}{4}}\geq-\frac{C}{T}.
\end{equation}
Letting $T\to +\infty$ in \eqref{eq-5.43} we obtain \eqref{eq-5.37}. The proof of Lemma \ref{lm5.16} is complete.
\end{proof}
Since $\CH^{n-2}(\mathrm{sing}\Sigma_\infty)=0$, we have
	\begin{corollary}\label{coro5.4.15}
			For any function $\varphi\in W^{1,2}_c(\Sigma_\infty)$ we have 
		\begin{equation}
			-\int_{\Sigma_\infty}(\varphi L\varphi)e^{-\frac{|x|^2}{4}}\geq 0.
		\end{equation}
	Here $ W^{1,2}_c(\Sigma_\infty)$	denotes the set of all functions $\varphi\in W^{1,2}(\Sigma_\infty)$ with compact
		support in $\Sigma_\infty$, and $L$ is the operator defined by \eqref{eq5.12} on the cone $\Sigma_\infty$.
	\end{corollary}
\begin{proof}[Proof of Theorem \ref{mul-one}]
     Suppose that there exists a senquence $\{t_i\}$ such that $\Sigma_{t_i}$ converges to $\Sigma_\infty\in \CC(N_0,n)$ with multiplicity greater than one. It follows from Corollary \ref{coro5.4.15} that $\Sigma_\infty$ is $F$-stable. However, it contradicts Lemma \ref{lm-b}. Thus, for the sequence of times $\{t_i\}$ in Lemma \ref{lm5.12}, the multiplicity of the convergence of $\{\Sigma_{t_i+t},-T<t<T\}$ is one and the convergence is smooth at $\mathrm{reg}(\Sigma_\infty)$.
\end{proof}

\subsection{Proof of Theorem \ref{thm0.1}}$\\$

For the convenience of readers, we copy down the statement of Theorem \ref{thm0.1} as follows.
\begin{theorem}
	For $n\ge 3$, let $\mathbf{x}:M^n\times [0,T)\to \mathbb{R}^{n+1})$ be a closed smooth embedded
	mean curvature flow in $\mathbb{R}^{n+1}$. Suppose that
	\begin{equation}
			\sup\limits_{M\times [0,T)}|H|(x,t)=\Lambda<\infty\quad \sup\limits_{t\in [0,T)} \mathrm{index}(M_t)=I<\infty,
	\end{equation}
		where $\mathrm{index}(M_t)$ is the number of negative eigenvalues of the operator $\Delta+|A|^2$ on $M_t$.
	Then there exist a smooth, possibly incomplete, hypersurface $M_T$ and a set $\mathcal{S}:=\overline{M_T}\setminus M_T$ satisfying the following properties.
    \begin{itemize}
        \item [(1)]$M_t$ smoothly converges to $M_T$ with multiplicity one. 
        \item [(2)] $\mathcal{S}$ has Minkowski dimension $\leq n-7$.
    In particular, for $3\leq n\leq 6$, $M_t$ does not blow up at time $T$. 

    \item [(3)]For any $p\in \mathcal{S}$ and any sequence of time $t_i\to T$, $\frac{1}{\sqrt{T-t_i}}(M_{t_i}-p)$(after taking a subsequence) converges to a stable  minimal cone with multiplicity one.
    \end{itemize}
\end{theorem}
\begin{proof}
By the compactness of varifold, there is a varifold $M_T$ such that $M_t\rightharpoonup M_T$ in the varifold sense as $t\to  T$. By Theorem \ref{thm3.1} and Proposition \ref{prop4.2}, $\CH^{\alpha}(\mathrm{sing}M_T)=0$ for any $\alpha>n-7$ and $M_t$ smoothly converges to $\mathrm{reg}M_T$ with multiplicity. We next show that the multiplicity is one. For any $x_0\in \mathrm{supp}M_T$, Corollary 3.6 of \cite{Hui90} implies that for all $t<T$ we have 
	\begin{equation}
		\begin{split}
			d(M_t,x_0)\leq 2\sqrt{T-t},
		\end{split}
	\end{equation}
	where $d(M_t,x_0)$ denotes the Euclidean distance from the point $x_0$ to the hypersurface $M_t$. We can rescale the flow $M_t$ by
	\begin{equation}
		\begin{split}
			s=-\log (T-t),\quad \tilde{M}_s=e^{\frac{s}{2}}\bigg(M_{T-e^{-s}}-x_0\bigg)
		\end{split}
	\end{equation}
	such that the flow $\{(\tilde{M}_s,\tilde{\mathbf{x}}(p,s)),-\log T\leq s<+\infty\}$ satisfies the following
	properties:
	\begin{itemize}
		\item [(1)]$\tilde{\mathbf{x}}(p,s)$ satisfies the equation
			\begin{equation}
			\bigg(\frac{\partial \tilde{\mathbf{x}}}{\partial t}\bigg)^\perp=-\bigg(H-\frac{1}{2}\langle \tilde{\mathbf{x}},\mathbf{n}\rangle \bigg)\mathbf{n};
		\end{equation}
		\item [(2)] the mean curvature of $\tilde{M}_s$ satisfies $\tilde{H}(p,s)\leq \Lambda e^{-\frac{s}{2}}$ for some $\Lambda>0$;
			\item [(3)] $\mathrm{index}(\tilde{M}_t)\leq I$;
		\item [(4)] $d(\tilde{M}_s,0)\leq 2$.
	\end{itemize}
	By Theorem \ref{mul-one}, for any sequence of times $s_i\to +\infty$  $\tilde{M}_{s_i}$ converges to a stable minimal cone with multiplicity one. In particular, if $x_0\in \mathrm{reg}M_T$, $\tilde{M}_{s_i}$ converges to a plane with multiplicity one. Hence, $\mathrm{sing}M_T$ has Minkowski dimension $\leq n-7$.
	
	For $3\leq n\leq 6$, since there is no stable cone other than a plane, $\tilde{M}_{s_i}$ converges smoothly to a plane passing through the origin with multiplicity one. Consider the heat-kernel-type function
	\begin{equation}
		\begin{split}
			\Phi_{(x_0,T)}(x,t)=\frac{1}{(4\pi (T-t))^{\frac{n}{2}}}e^{-\frac{|x-x_0|^2}{4(T-t)}},\quad \forall (x,t)\in M_t\times [0,T).
		\end{split}
	\end{equation}
	Huisken’s monotonicity formula(cf. Theorem 3.1 in \cite{Hui90}) implies that
	\begin{equation}
		\begin{split}
			\Theta(M_t,x_0,T)&:=\lim\limits_{t\to T}\int_{M_t}\Phi_{(x_0,T)}(x,t)d\mu_t\\
			&=\lim\limits_{s_i\to  +\infty}\frac{1}{(4\pi)^{\frac{n}{2}}}\int_{\tilde{M}_{s_i}}e^{-\frac{|x|^2}{4}}d\tilde{\mu}_{s_i}=1,
		\end{split}
	\end{equation}
	which implies that $(x_0,T)$ is a regular space-time point by Theorem 3.1 of White \cite{Whi05}. Thus, the unnormalized mean curvature flow $\{(M,\mathbf{x}(t)),0\leq t<T\}$ cannot develop a singularity at
	 $(x_0,T)$. The theorem is proved.
\end{proof}

\appendix

\section{A sheeting theorem}\label{secA}
In this section, we give a proof of Proposition \ref{thm3.1}.
\begin{proposition}
	If $\{M^n_k\}\subset \mathbb{R}^{n+1}$ is a sequence of complete, connected and embedded hypersurfaces with  $\CH^{n-2}(\operatorname{sing} M_k)=0$ and satisfying
	\begin{itemize}
		\item [(a)] $\sup\limits_{k\to \infty}\|\textbf{H}\|_{L^\infty(M_k)}\leq \Lambda_1$,
		\item [(b)] $\sup\limits_{x\in M_k}\sup\limits_{r>0}\frac{\mathcal{H}^n(M_k\cap B_r(x))}{\omega_n r^n}\leq  \Lambda_2$,
		\item [(c)] $\mathrm{index}(M_k)\leq I$.
	\end{itemize}
	for some fixed constants $\Lambda_1,\Lambda_2>0, I\in \mathbb{N}$. Then up to a subsequence, there exists a varifold $V$ such that $M_k$ converges to $V$ as varifold and such that
	\begin{equation}
		\begin{split}
			\mathrm{spt}V=\bar{M}
		\end{split}
	\end{equation} 
	where $M$ is a connected and oriented $C^{1,\beta}$ hypersurface in $\mathbb{R}^{n+1}$ and
	\begin{itemize}
		\item [$\mathrm(1)$] $\|\textbf{H}\|_{L^\infty(M)}\leq \Lambda_1$,
		\item [$\mathrm(2)$] $\sup\limits_{x\in M}\sup\limits_{r>0}\frac{\mathcal{H}^n(M\cap B_r(x))}{\omega_n r^n}\leq  \Lambda_2$,
		\item [$\mathrm(3)$] $\CH^{n-2}(\mathrm{sing}M)=0$.
	\end{itemize}
	We have that the convergence is in  $C^{1,\beta}$ topology for all $x\in \mathrm{reg}M\setminus\CY$ and $\beta\in (0,\beta_0)$, where $\CY\subset \bar{M}$ is a discrete set  such that $|\CY|\leq I$.
	
	Moreover, if  the convergence is $C^2$ in $\mathrm{reg}M\setminus\CY$, we have
	\begin{equation}
		\begin{split}
			\mathrm{index}(M)\leq I \text{ and }	\CH^{\alpha}(\mathrm{sing}M)=0	 \text{ for any }\alpha>n-7.
		\end{split}
	\end{equation}
\end{proposition}

\begin{proof}

 \textit{Step 1: $I=0$.}
 
 By Allard’s compactness theorem, there exist an integral varifold $V$ and a subsequence
 $\{M_{k'}\}$ such that
 \[
 M_{k'} \rightharpoonup V
 \]
 in the varifold sense. Denote $M=\operatorname{spt} V$. Then $M_{k'}$ converges to $M$
 in the local Hausdorff distance. For any $x\in \operatorname{reg} M$, the tangent cone
 of $V$ at $x$ is a plane. Consequently, there exist $\rho_x>0$ and $k_0\in\mathbb{N}$
 such that $M_{k'}$ satisfies the conditions of Lemma~\ref{lm3.2}, after scaling, for all $k'>k_0$. By \eqref{eq3.7}, it follows that $M_{k'}$
 converges to $M$ in the $C^{1,\beta}$ sense at $x$.
 
 \medskip
 \noindent\textit{Step 2: $I\neq 0$.}
 
 Define
 \begin{equation}
 	\tilde{\mathcal{Y}}
 	=
 	\Bigl\{
 	x\in \operatorname{spt} V :
 	\text{for any } r>0,\;
 	\liminf_{k\to\infty}
 	\operatorname{index}(M_k\cap B_r(x))\ge 1
 	\Bigr\}.
 \end{equation}
 
 We claim that $|\tilde{\mathcal{Y}}|\le I$. Indeed, otherwise there would exist
 $I+1$ distinct points
 $\{y_i\}_{i=1}^{I+1}\subset \tilde{\mathcal{Y}}$. One can then choose radii
 $\{r_i\}_{i=1}^{I+1}$ such that
 \[
 B_{r_i}(y_i)\cap B_{r_j}(y_j)=\emptyset
 \quad \text{for } i\neq j,
 \]
 and
 \[
 \operatorname{index}(M_k\cap B_{r_i}(y_i))\ge 1
 \quad \text{for } 1\le i\le I+1,
 \]
 for all sufficiently large $k$. By Lemma~\ref{lm1.7}, this implies
 \[
 \operatorname{index}(M_k)\ge I+1
 \]
 for $k$ large enough, contradicting the assumed index bound.
 
 For any $y\in \operatorname{supp}V\setminus \tilde{\mathcal{Y}}$, there exists $r>0$
 such that $M_k$ is stable in $B_r(y)$ for all sufficiently large $k$. By Lemma \ref{lm3.2}, $M_k$ converges to $\operatorname{reg} M$ in the
 $C^{1,\beta}$ sense in $B_{r/2}(y)$, and
 \[
 \mathcal{H}^{n-2}\bigl(\operatorname{sing} M\cap B_{r/2}(y)\bigr)=0.
 \]
 A standard covering argument then yields $ \mathcal{H}^{n-2}(\operatorname{sing} M)=0.$
 
 If the convergence is in fact $C^2$, Lemma~\ref{index-bound} implies
 \[
 \operatorname{index}(M)
 \le \liminf_{k\to\infty}\operatorname{index}(M_k)
 \le I.
 \]
 Finally, combining Lemma~\ref{lm3.2} with Federer’s inductive dimension-reduction
 argument \cite{Fer70}, adapted to the varifold setting (see the proof of
 Theorem~2 in \cite{SS81}), we conclude that
 \[
 \mathcal{H}^{\alpha}(\operatorname{sing} M)=0
 \quad \text{for all } \alpha>n-7.
 \]
 
 The theorem follows by choosing $ \mathcal{Y}=\operatorname{reg} M\setminus \tilde{\mathcal{Y}}.$
\end{proof}

\begin{remark}
Naber-Valtorta \cite{NV20}	showed that $\mathcal{H}^{n-7}( \mathrm{sing} M\cap B^{n+1}_{\rho_0/2}(0))\leq C\rho_0^{n-7}$. 
\end{remark}

\end{document}